\documentclass[12pt, a4paper]{amsart}
\usepackage[a4paper, margin=3cm]{geometry}
\usepackage{amssymb,amsfonts}

\usepackage{color}
\newtheorem{Theorem}{Theorem}[section]

\newtheorem{Lemma}[Theorem]{Lemma}
\newtheorem{Corollary}[Theorem]{Corollary}

\newtheorem{Proposition}[Theorem]{Proposition}
\theoremstyle{definition}
\newtheorem{Definition}[Theorem]{Definition}
\newtheorem{Remark}[Theorem]{Remark}

\def\C{\mathbb C}
\def\E{\mathbb E}

\def\N{\mathbb N}

\def\R{\mathbb R}

\def\Z{\mathbb Z}

\newcommand{\CF}{\mathcal{F}}

\newcommand{\CH}{\mathcal{H}}

\newcommand{\CJ}{\mathcal{J}}

\newcommand{\CP}{\mathcal{P}}

\newcommand{\CS}{\mathcal{S}}

\newcommand{\CV}{\mathcal{V}}
\newcommand{\CW}{\mathcal{W}}

\newcommand{\CZ}{\mathcal{Z}}

\def\be{\begin{equation}}
\def\ee{\end{equation}}
\def\bt{\begin{Theorem}}
\def\et{\end{Theorem}}
\def\bi{\begin{itemize}}
\def\ei{\end{itemize}}
\def\bea{\begin{eqnarray}}
\def\eea{\end{eqnarray}}
\def\beast{\begin{eqnarray*}}
\def\eeast{\end{eqnarray*}}
\def\ben{\begin{enumerate}}
\def\een{\end{enumerate}}

\def\bi{\bibitem}

\newcommand{\norm}[1]{\left\Vert#1\right\Vert}

\def\lan{{\langle}}
\def\ran{{\rangle}}

\newcommand{\half}{{\frac{1}{2}}}

\renewcommand{\MR}[1]{} 

\DeclareMathOperator{\dom}{\mathfrak{D}}
\DeclareMathOperator{\fhalf}{\frac{1}{2} }
\newcommand{\abs}[1]{\left\vert#1\right\vert}

\allowdisplaybreaks[4]

\begin{document}
\title[$q$-deformed  Araki-Woods von Neumann algebras]{Generator masas in $q$-deformed  Araki-Woods von Neumann algebras
and factoriality}

\author[Bikram]{Panchugopal Bikram}
\author[Mukherjee]{Kunal Mukherjee}

\address{Department of Statistics and Mathematics,
Indian Statistical Institute, No. 203, Barrackpore Trunk Road, Kolkata -  700108, India.}
\address{ Department of Mathematics, IIT Madras, Chennai - 600036 , India.}
\email{pg.math@gmail.com, kunal@iitm.ac.in}


\begin{abstract} 
To any strongly continuous orthogonal representation of $\R$ on a real Hilbert space $\CH_\R$, Hiai
constructed $q$-deformed Araki-Woods von Neumann algebras for $-1< q< 1$, which are $W^{\ast}$-algebras  
arising from non tracial representations of the $q$-commutation relations, the latter yielding an interpolation
between the Bosonic and Fermionic statistics. We prove that if the orthogonal representation is not ergodic then
these von Neumann algebras are factors whenever $dim(\CH_\R)\geq 2$ and $q\in (-1,1)$. In such case, the centralizer of the $q$-quasi free state has trivial relative commutant. In the process, we study `generator masas' in these factors and establish that they are strongly mixing. 
\end{abstract} 
\maketitle

\section{Introduction}

In free probability, Voiculescu's $C^*$-free Gaussian functor associates a canonical $C^{*}$-algebra denoted by $\Gamma(\CH_\R)$ to a real Hilbert space $\CH_\R$, the former being generated by $s(\xi), \xi\in \CH_\R$, where each $s(\xi)$ is the sum of creation and annihilation operators on the full Fock space of the complexification of $\CH_\R$. The associated von Neumann algebra $\Gamma(\CH_\R)^{\prime\prime}$ is isomorphic to $L(\mathbb{F}_{dim( \CH_\R)})$ and is the central object in the study of free probability $($see \cite{DVN} for more on the subject$)$. In the literature, there are three interesting types of deformations of Voiculescu's free Gaussian functor each of which has a real Hilbert space $\CH_\R$ as the initial input data: $(i)$ the $q$-Gaussian functor due to Bo$\overset{.}{\text{z}}$ejko and Speicher for $-1<q<1$ (see \cite{BS}), $(ii)$ a functor due to Shlyakhtenko (see \cite{Shlyakhtenko}) which is a  free probability
analog of the construction of quasi free states on the CAR and CCR algebras and $(iii)$ the third one is a combination of the first two and is due to Hiai (see \cite{Hiai}); the associated von Neumann algebras are respectively called Bo$\overset{.}{\text{z}}$ejko-Speicher factors $($or $q$-Gaussian von Neumann algebras$)$, free Araki-Woods factors and $q$-deformed Araki-Woods von Neumann algebras.

Historically, for the first time, Frisch and Bourret in \cite{FB70} had considered operators satisfying the $q$-canonical 
commutation relations:
\begin{align*}
l(e)l(f)^* - q l(f)^*l(e) = \langle e, f\rangle I, \text{ }-1<q<1.
\end{align*}
The existence of such operators on an `appropriate Fock space' was proved by Bo$\overset{.}{\text{z}}$ejko and Speicher in \cite{BS}  and these operators have importance in particle statistics \cite{Gre90,GGG93}. Since then many experts have studied the $q$-Gaussian von Neumann algebras. Structural properties of the $q$-Gaussian algebras have been studied in \cite{avsec,BS,BKS,Dab14,ER,nou,sniady,Shlyakhtenko2,Shlyakhtenko3}. A short summary of the results obtained in these studies are as follows. For $dim(\CH_\R)\geq 2$, the $q$-Gaussian von Neumann algebras $\Gamma_q(\CH_\R)$ are non-injective, solid, strongly solid, non $\Gamma$ factors with $w^*$-completely contractive approximation property. Further, $\Gamma_{q}(\CH_\R)\cong L(\mathbb{F}_{dim(\CH_\R)})$ for values of $q$ sufficiently close to zero \cite{GS14}.

The Shlyakhtenko functor in \cite{Shlyakhtenko} associates a $C^*$-algebra $\Gamma(\CH_\R, U_t)$ to a pair $( \CH_\R, U_t)$, where $\CH_\R$ is a real Hilbert space and $(U_t)$ is a strongly continuous real orthogonal representation of $\R$ on $\CH_\R$. The von Neumann algebras $\Gamma(\CH_\R, U_t)^{\prime\prime}$ obtained this way i.e., the free Araki-Woods von Neumann algebras are full factors of type $\rm{III}_{\lambda}$, $0<\lambda\leq 1$, when $(U_t)$ is non trivial and $dim(\CH_\R)\geq 2$ \cite{Shlyakhtenko}. These von Neumann algebras are type $\rm{III}$ counterparts of the free group factors. In short, they satisfy complete metric approximation property, lack Cartan subalgebras, are strongly solid, and, they satisfy Connes' bicentralizer problem when they are type $\rm{III}_1$ $($see \cite{Ho09, HR11,BHV15}$)$. They have many more interesting properties.

The third functor mentioned above is the $q$-deformed functor due to Hiai for $-1<q<1$ (see \cite{Hiai}). Hiai's functor is the main topic of this paper. It is a combination of Bo$\overset{.}{\text{z}}$ejko-Speicher's functor and Shlyakhtenko's functor. This functor, like the Shlyakhtenko's functor, associates a $C^*$-algebra $\Gamma_q(\CH_\R, U_t)$ to a pair $( \CH_\R, U_t)$, where $\CH_\R$ is a real Hilbert space and $(U_t)$ is a strongly continuous orthogonal representation of $\R$ on $\CH_\R$ as before. The associated von Neumann algebras in this construction depend on $q\in (-1,1)$ and are represented in standard form on `twisted full Fock spaces' that carry the spectral data of $(U_t)$ and connects it to the modular theory of this particular standard representation in a manner such that the canonical creation and annihilation operators satisfy the $q$-canonical commutation relations of Frisch and Bourret. Hiai's functor coincides with Bo$\overset{.}{\text{z}}$ejko-Speicher's functor when $(U_t)$ is trivial and also coincides with Shlyakhtenko's functor when $q = 0$. Note that $\Gamma_q(\CH_\R, U_t)^{\prime\prime}$ is abelian when $dim(\CH_\R)=1$, so the 
situation becomes interesting when $dim(\CH_\R)\geq 2$.

Assume $dim(\CH_\R)\geq 2$. Unlike the free Araki-Woods factors, not much is known about the $q$-deformed Araki-Woods von Neumann algebras. Hiai proved amongst other things that when the almost periodic part of $(U_t)$ is infinite dimensional, the centralizer of the $q$-quasi free state $($vacuum state$)$ has trivial relative commutant and thus decided factoriality of the ambient von Neumann algebra $\Gamma_q(\CH_\R, U_t)^{\prime\prime}$ $($Thm. 3.2, \cite{Hiai}$)$. Thus, he was also able to decide the type of these factors under the same hypothesis imposed $($Thm. 3.3, \cite{Hiai}$)$. He also exhibited non injectivity of $\Gamma_q(\CH_\R, U_t)^{\prime\prime}$ depending on the `thickness of the spectrum of the analytic generator of $(U_t)$' $($Thm. 2.3, \cite{Hiai}$)$. Recently, Nelson generalized the techniques of free monotone transport originally developed in \cite{GS14} beyond the tracial case. Using this powerful tool he proved that $\Gamma_q(\CH_\R, U_t)^{\prime\prime}\cong \Gamma_0(\CH_\R, U_t)^{\prime\prime}$ $($the latter being the free Araki-Woods factors$)$ around a small interval centred at $0$, and hence decide factoriality $($Thm. 4.5, 4.6, \cite{Nel15}$)$. Thus, even factoriality of $\Gamma_q(\CH_\R, U_t)^{\prime\prime}$ is not known to hold in general. The main purpose of this paper is to investigate the factoriality of $\Gamma_q(\CH_\R, U_t)^{\prime\prime}$.
The main result in this paper is the following:

\begin{Theorem}\label{MainFactorialityTheorem}
For any strongly continuous orthogonal representation $t\mapsto U_t$, of $\R$ on a separable real Hilbert space $\CH_\R$ with $dim(\CH_\R)\geq 2$ and for all $q\in (-1,1)$, the $q$-deformed Araki-Woods von Neumann algebras $\Gamma_q(\CH_\R,U_t)^{\prime\prime}$ are factors, if there exists a unit vector $\xi_0\in \CH_\R$ such that $U_t\xi_0=\xi_0$ for all $t\in\R$.
\end{Theorem}

Our proof is primarily motivated by the main result in \cite{ER} on masas. The second observation is, if a finite von Neumann algebra contains a diffuse masa so that the orthocomplement of the associated Jones' projection $($with respect to a faithul normal tracial state$)$ as a bimodule over the masa is a direct sum of coarse bimodules, then  the ambient von Neumann algebra must be a factor. Thus, our proof depends on singular masas $($and this is natural as we are dealing with algebras which are close to free group factors$)$. So, our techniques are more close to understanding the measure-multiplicity invariant of a masa that was introduced in \cite{DSS06}. The masas that we work with lie in the centralizer of the $q$-quasi free state. We call these as \emph{generator masas}, as these masas are indeed the analogue of generator masas in the free group factors. The generator masas in the free group factors have vigorous mixing properties. So, to compare, we investigate mixing properties of generator masas in $\Gamma_q(\CH_\R,U_t)^{\prime\prime}$ and show that the left-right measure of these masas $($see \cite{Muk09} for Defn.$)$ are Lebesgue absolutely continuous. 

Now we describe the layout of the paper. In \S2, we collect all the necessary material that is needed to address the problem. This section contains an account of Hiai's construction, associated modular theory, description of the commutant and other technical details. A convenient description of the centralizer of the $q$-quasi free state is required. The centralizer depends entirely on the almost periodic component of $(U_t)$ and its GNS space is described in Thm. \ref{CentraliserDescribe}. In \S4, we investigate the properties of the generator abelian algebras which are indispensable ingredients in our arguments. In Thm. \ref{eigen-vector}, we establish that a canonical self-adjoint generator of $\Gamma_q(\CH_\R,U_t)^{\prime\prime}$ generates a diffuse abelian algebra $($generator masa$)$ having conditional expectation that preserves the vacuum state if and only if the generator lies in the centralizer of $\Gamma_q(\CH_\R,U_t)^{\prime\prime}$ with respect to the same state.

We begin \S5 by making a short account on how to regard a GNS space of an arbitrary von Neumann algebra equipped with a faithful normal state as a standard bimodule over a masa, when the masa comes from the centralizer of the associated state. 
We also discuss strong mixing of masas $($lying inside the centralizer$)$ with respect to a particular faithful normal state and also highlight on calculating left-right measures of masas. In Thm. \ref{stronglymixingmasa} and Thm. \ref{masa}, we show that for a generator algebra $($masa$)$ in $\Gamma_q(\CH_\R,U_t)^{\prime\prime}$ that possess conditional expectation preserving the vacuum state, the left-right measure is indeed \emph{Lebesgue absolutely continuous for all} $q\in (-1,1)$. This justifies the term `generator masa'. This statement is an indication that $\Gamma_q(\CH_\R,U_t)^{\prime\prime}$ will share many properties of the free group factors even when $q$ is away from $0$ $($the case when $q$ is close to $\pm 1$ is probably more interesting from the point of view of physics$)$ and is a reflection of a deep theorem of Voiculescu on the subject \cite{Voi96}. It readily follows that if the fixed point subspace of $(U_t)$ is at least two dimensional, then the centralizer of the vacuum state has trivial relative commutant and hence $\Gamma_q(\CH_\R,U_t)^{\prime\prime}$ is a factor $($Cor. \ref{Trivialrelcommutant}$)$.

In \S6, we establish factoriality of $\Gamma_q(\CH_\R,U_t)^{\prime\prime}$ in Thm. \ref{wmixingfactor} and Thm. \ref{Factor1eigenvalue}, when $dim(\CH_\R)\geq 2$ and $q\in (-1,1)$, in the case when $(U_t)$ is not ergodic or has a non trivial weakly mixing component. In \S\ref{StructureofCentralizer}, we extend the statement of Cor. \ref{Trivialrelcommutant} in Thm. \ref{TrivialRelativeCommutant} to show that the centralizer of the vacuum state has trivial relative commutant when $dim(\CH_\R)\geq 2$, the fixed point subspace of $(U_t)$ is at least one dimensional  and the dimension of the almost periodic part of $(U_t)$ is at least two dimensional. Finally, we characterize the type of the factors obtained via Hiai's construction in Thm. \ref{WeakmixingtoIII_1} and Thm. \ref{AlmostPeriodicCase} under the assumption that $(U_t)$ is almost periodic with a non trivial fixed point or has a weakly mixing component. The results are analogous to the ones found in Thm. $3.3$ \cite{Hiai}.

$ $\\
\noindent\textbf{Acknowledgements}: The authors thank Roland Speicher, Pitor {\'S}niady and Todd Kemp for their help with references. Special thanks to Fumio Hiai for sending his papers promptly on request which were not available to us. The first author is grateful to IIT Madras and Hausdorff Research Institute for their warm hospitality where much of this work was completed. The second author thanks Arindama Singh, Jon Bannon and Jan Cameron for helpful discussions.


\section{Preliminaries}\label{Araki-Woods}

In this section, we collect some well known facts about the $q$-deformed Araki-Woods von Neumann algebras constructed by Hiai in \cite{Hiai} that will be indispensable for our purpose. For detailed exposition, we refer the interested readers to \cite{Hiai}. As a convention, all Hilbert spaces in this paper are separable, all von Neumann algebras have separable preduals and inner products are linear in the \emph{second variable}.

\subsection{Hiai's Construction}\label{ConsHiai} Let $\CH_{\R}$ be a real Hilbert space and let $t\mapsto U_t$, $t\in \mathbb{R}$,
be a strongly continuous orthogonal representation of $\R$ on $\CH_{\R}$. 
Let  $\CH_\C=\CH_{\R}\otimes_\R \C$ denote the complexification of $\CH_{\R}$.
Denote the inner product and norm on $\CH_\C $ by $\langle \cdot , \cdot \rangle_{\CH_\C} $ 
and $\norm{\cdot}_{\CH_\C}$ respectively. Identify $\CH_{\R}$ in $\CH_{\C}$ by $\CH_{\R} \otimes 1 $. 
Thus, $\CH_\C =  \CH_{\R} + i \CH_{\R}$, and as  a real Hilbert space the inner product of 
$\CH_{\R} $ in $\CH_\C $ is given by $\Re \langle \cdot, \cdot \rangle_{\CH_\C} $.
Consider the bounded anti-linear operator $\mathcal{J}: \CH_\C \rightarrow \CH_\C  $ given  by 
$\mathcal{J}(\xi + i \eta )= \xi - i \eta$, $\xi, \eta \in \CH_{\R} $,
and note that $\mathcal{J}\xi = \xi$ for $\xi\in \CH_{\R}$. Moreover,
\begin{align*}
\langle \xi, \eta \rangle_{\CH_\C} = \overline {\langle \eta , \xi \rangle}_{\CH_\C}  = \langle \eta , \mathcal{J} \xi \rangle_{\CH_\C}, \text{ for all } \xi \in \CH_{\C} , \eta \in \CH_{\R}.
\end{align*}
Linearly extend the flow $t\mapsto U_t$ from $\CH_{\R} $ to a strongly continuous one parameter group of unitaries 
in $\CH_{\C}$  and denote the extensions by $U_t$ for each $t$ with abuse of notation.
Let $A$ denote the analytic generator and $H$ the associated Hamiltonian of the extension. Then $A$ is positive,  nonsingular and self-adjoint, while $H$ is self-adjoint. Since $\CH_{\R}$ reduces $U_{t}$ for all $t\in \mathbb{R}$, so $\CH_{\R}$ reduces $iH$ as well. Denoting $\mathfrak{D}(\cdot)$ to be the domain of an $($unbounded$)$ operator, one notes that $\mathfrak{D}(H) = \mathfrak{D}(iH)$ and $H$ maps $\mathfrak{D}(H)\cap \CH_{\R}$ into $i\CH_{\R}$.  
It follows that $\CJ H = -H\CJ$ and $\CJ A = A^{-1}\CJ$.

Introduce a new inner product on $\CH_{\C}$ by $\langle \xi, \eta \rangle_U = \langle \frac{2}{1+ A^{-1}} \xi, \eta \rangle_{\CH_\C}$, $\xi, \eta \in \CH_{\C}$, and let $\norm{\cdot}_{U}$ denote the associated norm on $\CH_{\C}$. Let $\CH$ denote the complex Hilbert space obtained by completing $(\CH_{\C}, \norm{\cdot}_{U})$.  The inner product and norm of $\CH$ will respectively be denoted by $\langle\cdot,\cdot\rangle_U$ and $\norm{\cdot}_{U}$ as well. Then, $(\CH_{\R}, \norm{\cdot}_{\CH_\C})\ni \xi \overset{\imath}\mapsto\xi \in (\CH_{\C}, \norm{\cdot}_{U})\subseteq (\CH,\norm{\cdot}_{U})$, is an isometric embedding of the real Hilbert space $\CH_\R$ in $\CH$ $($in the sense of \cite{Shlyakhtenko}$)$. With abuse of 
notation, we will identify $\CH_\R$ with its image $i(\CH_\R)$. Then, $\CH_{\R}\cap i\CH_{\R}=\{0\}$ and $\CH_{\R}+ i\CH_{\R}$ is dense in $\CH$ $($see pp. 332 \cite{Shlyakhtenko}$)$.

It is now appropriate to record a subtle point which will be crucial in our attempt to describe the centralizers of the 
$q$-deformed Araki-Woods von Neumann algebras. As $A$ is affiliated to $vN(U_t:t\in \R)$, so note that 
\begin{align}\label{Liftisunitary}
\langle U_t\xi, U_t\eta\rangle_U=\langle \xi,\eta\rangle_U, \text{ for }\xi,\eta\in \CH_\C.
\end{align}
Consequently, $(U_t)$ extends to a strongly continuous unitary representation $(\widetilde{U}_t)$ of $\R$ on $\CH$. 
Let $\widetilde{A}$ be the analytic generator associated to $(\widetilde{U}_t)$, which is obviously an extension of $A$. From the definition of $\langle\cdot,\cdot\rangle_U$ on $\CH_\C$, it follows that if $\mu$ is the spectral measure of $A$, then $\nu=f\mu$ is the spectral measure of $\widetilde{A}$, where $f(x)=\frac{2x}{1+x}$ for $x\in\R_{\geq 0}$, and by the spectral theorem $($direct integral form$)$, the multiplicity functions in the associated direct integrals remain the same. Note that $L^{2}(F, \mu_{|F})\subseteq L^{2}(F,\nu_{|F})$ for all Borel subsets $F$ of $(0,\infty)$. But, as $f$ is increasing, it follows that $L^{2}(F, \mu_{|F})=L^{2}(F,\nu_{|F})$ $($as a vector space$)$ when $F\subseteq[\lambda,\infty)$ is measurable for all $\lambda>0$. Moreover, $0<\lambda$ is an atom of $\mu$ if and only if it is an atom of $\nu$.  Thus, if $E_A$ and $E_{\widetilde{A}}$ denote the associated projection-valued spectral measures, then $E_{A}([\lambda,\infty))(\CH_\C)=E_{\widetilde{A}}([\lambda,\infty))(\CH)$ and $E_{A}(\lambda)(\CH_\C)=E_{\widetilde{A}}(\lambda)(\CH)$ for all $\lambda>0$. We record the following in the form of a proposition. 

\begin{Proposition}\label{Eigenvector}
Any eigenvector of $\widetilde{A}$ is an eigenvector of $A$ corresponding to the same eigenvalue. 
\end{Proposition}

Since the spectral data of $A$ and $\widetilde{A}$ $($and hence of $(U_t)$ and $(\widetilde{U}_t))$ are essentially the same, and $\widetilde{U}_t,\widetilde{A}$ are respectively extensions of $U_t,A$ for all $t\in \R$, so we would now write 
$\widetilde{A}=A$ and $\widetilde{U}_t=U_t$ for all $t\in\R$. This abuse of notation will cause no confusion; on occasions 
where we need to differentiate the two, we will notify.

Given a complex Hilbert space and $ -1 < q < 1 $, the notion of $q$-Fock space $\CF_q(\cdot)$ was introduced in \cite{BS}. The $q$-Fock space $\CF_q(\CH)$ of $\CH$ is constructed as follows. Let $\Omega$ be a distinguished unit vector in $\C$ usually referred to as the vacuum vector. Denote $\CH^{\otimes 0}= \C\Omega$, and, for $n\geq 1$, let $\CH^{\otimes n} = \text{ span}_{\C}\{\xi_1 \otimes \cdots \otimes \xi_n : \xi_i \in \CH \text{ for }1\leq i\leq n\}$ denote the algebraic tensor products. Let $\CF_{fin}(\CH)=\text{ span}_{\C}\{\CH^{\otimes n}: n\geq 0 \}$. For $n,m\geq 0$ and $f=\xi_1 \otimes \cdots \otimes \xi_n \in \CH^{\otimes n}$, $g= \zeta_1\otimes \cdots \otimes \zeta_m\in \CH^{\otimes m}$, the association 
\begin{align}\label{qFock}
\langle f, g \rangle_q = \delta_{m, n} 
\sum_{ \pi \in S_n } q^{ i(\pi) } \langle \xi_1, \zeta_{\pi(1) }\rangle_U \cdots\langle \xi_n, \zeta_{\pi(n) } \rangle_U, 
\end{align} 
where $i(\pi)$ denotes the number of  inversions of the permutation $\pi \in S_n $, defines a positive definite sesquilinear form on  $ \CF_{fin}(\CH)$ and the $q$-Fock space $\CF_q(\CH) $ is the completion of $ \CF_{fin}(\CH)$ with respect to the norm $\norm{\cdot}_{q}$ induced by $ \langle \cdot, \cdot \rangle_q $.

For $n\in\N$, let $\CH^{\otimes_q n}=\overline{\CH^{\otimes n}}^{\norm{\cdot}_q}$. For our purposes, it is important to note that $\langle\cdot,\cdot\rangle_q$ and $\langle\cdot,\cdot\rangle_0$ are equivalent on $\CH^{\otimes n}$ and $\langle\cdot,\cdot\rangle_0$ is the inner product of the standard tensor product. Thus,  rephrasing and combining two lemmas of \cite{BS} one has the following.

\begin{Lemma}\label{Equivalent}
The map $id: (\CH^{\otimes n},\norm{\cdot}_q)\rightarrow (\CH^{\otimes n},\norm{\cdot}_0)$, given by $id(\xi_1\otimes\cdots\otimes\xi_n)=(\xi_1\otimes\cdots\otimes\xi_n)$, where $\xi_i\in\CH$, $1\leq i\leq n$, extends uniquely to a bounded and invertible linear map $S:(\CH^{\otimes_q n},\norm{\cdot}_q)\rightarrow (\CH^{\otimes_0 n},\norm{\cdot}_0)$ for $-1<q<1$. 
\end{Lemma}

\begin{proof}
Following \cite{BS}, every $\pi\in S_n$ induces an unitary operator on $\CH^{\otimes_0 n}$ given by $U_\pi(\xi_1\otimes\cdots\otimes\xi_n)=\xi_{\pi(1)}\otimes\cdots\otimes\xi_{\pi(n)}$, $\xi_i\in\CH$, $1\leq i\leq n$. 
Let $P_q=\sum_{\pi\in S_n} q^{i(\pi)}U_{\pi}$. Then, $P_q\in \textbf{B}(\CH^{\otimes_0 n})$ and by Lemma 3 and Lemma 4 of \cite{BS}, $P_q$ is strictly positive for $-1<q<1$ and $\langle f,g\rangle_q =\langle f,P_q g\rangle_0$ for all $f,g\in 
\CH^{\otimes n}$. Consequently, $P_q$ is injective and hence invertible. It follows that 
\begin{align*}
\frac{1}{\norm{P_q^{-\half}}}\norm{f}_0\leq \norm{f}_q\leq \norm{P_q}^{\half}\norm{f}_0,\text{ for }f\in \CH^{\otimes n}.
\end{align*}
The rest is obvious.
\end{proof}

The following norm inequalities will be crucial (c.f.  \cite{BKS}, \cite{BS},  and \cite{ER}):

\noindent$\bullet$ If $\xi\in \CH$ and ${\norm{\xi}}_U = 1$, then 
\begin{align}\label{Normelt}
{\norm{\xi^{\otimes n}}}_q^2 = [n]_q!, 
\end{align}
where $[n]_q := 1+ q+ \cdots +  q^{(n-1)}$, $[n]_q! :=\prod_{j=1}^{n} [j]_q, \text{ for }n\geq 1$,    
and $[0]_q := 0$, $[0]_q! := 1$ by convention.

\noindent$\bullet$ If $\xi_1,\cdots, \xi_n , \xi \in \CH$ with $ {\norm{\xi_j}}_U={\norm{\xi}}_U = 1$ for all $1\leq j\leq n$, then the following estimate holds:
\begin{align}\label{normestimate}
{\norm{\xi_1 \otimes \cdots \otimes \xi_n \otimes \xi^{\otimes m}}}_q \leq C_q^{\frac{n}{2}} \sqrt{ [m]_q!}, \text{ }m\geq 0,
\end{align}
where $C_q = \prod_{ i= 1 }^\infty \frac{1}{(1-|q|^i)} $.

For $\xi\in \CH $, the left $q$-creation and $q$-annihilation operators on $\CF_q(\CH) $ are respectively defined by:
\begin{align}\label{Leftmult}
&c_q(\xi)\Omega  = \xi, \\ 
\nonumber&c_q(\xi) (\xi_1 \otimes \cdots\otimes \xi_n) =\xi \otimes  \xi_1 \otimes \cdots \otimes \xi_n,\\
\nonumber&\text{and},\\
\nonumber&c_q(\xi)^*\Omega  = 0,\\ 
\nonumber&c_q(\xi)^* (\xi_1 \otimes \cdots\otimes \xi_n) = \sum_{i = 1}^n{q^{i-1}}
\langle   \xi , \xi_i\rangle_U \xi_1 \otimes \cdots \otimes \xi_{i-1} \otimes \xi_{i+1} \otimes \cdots \otimes \xi_n, 
\end{align}
where $\xi_1 \otimes \cdots\otimes \xi_n\in \CH^{\otimes_q n}$ for $n\geq 1$.
The operators $c_q(\xi)$ and $c_q(\xi)^* $ are bounded on $ \CF_q(\CH) $ and they are 
adjoints of each other. Moreover, they satisfy the following $q$-commutation relations:
\begin{align*}
c_q(\xi)^* c_q(\zeta) -q c_q(\zeta) c_q(\xi)^* = \langle \xi, \zeta  \rangle_U 1, \text{ for all } \xi,\zeta\in \CH.
\end{align*}

The following observation will be crucial for our purpose. 

\begin{Lemma}\label{SplitAdjoint}
Let $\xi, \xi_i, \eta_j \in \CH$, for $1\leq i\leq n$, $1\leq j\leq m$. Then, 
\begin{align*} 
c_q&(\xi)^* \bigg( (
\xi_1  \otimes \cdots \otimes \xi_n) \otimes (\eta_1 \otimes \cdots \otimes \eta_m)\bigg)\\
&= \bigg(c_q(\xi)^*
(\xi_1\otimes \cdots \otimes \xi_n)\bigg) \otimes (\eta_1 \otimes \cdots \otimes \eta_m)\\
&\indent\indent\indent\indent\indent+q^n 
(\xi_1  \otimes \cdots \otimes \xi_n) \otimes \bigg(c_q(\xi)^*(\eta_1 \otimes \cdots \otimes \eta_m) \bigg).
\end{align*} 
\end{Lemma}
\begin{proof}
The proof follows easily from Eq. \eqref{Leftmult}. We leave it as an exercise.
\end{proof}

Following \cite{Shlyakhtenko} and \cite{Hiai}, consider the $C^*$-algebra 
$\Gamma_q( \CH_\R, U_t) =: C^{*}\{ s_q(\xi) : \xi \in \CH_\R \}$ and the von Neumann algebra
$\Gamma_q( \CH_\R, U_t)^{\prime\prime}$, where 
\begin{align*}
s_q(\xi)  = c_q(\xi) + c_q(\xi)^* , \text{ }\xi \in \CH_\R.
\end{align*}  
$\Gamma_q( \CH_\R, U_t)^{\prime\prime}$ is known as the $q$-\textit{deformed Araki-Woods von Neumann algebra}  (see \cite[\S3]{Hiai}).  The vacuum state $\varphi_{q, U}:= 
\langle \Omega, \cdot\text{ } \Omega\rangle_q$ $($also called the $q$-\textit{quasi free state}$)$, is a faithful normal state of $\Gamma_q( \CH_\R, U_t)^{\prime\prime}$ and $\CF_{q}(\CH)$ is the GNS Hilbert space of $\Gamma_q( \CH_\R, U_t)^{\prime\prime}$ associated to $\varphi_{q,U}$. Thus, $\Gamma_q( \CH_\R, U_t)^{\prime\prime}$ acting on $\CF_{q}(\CH)$ is in standard form [Hag75].  

\textit{Making slight violation of the traditional notations, we would use the symbols $\langle \cdot,\cdot\rangle_q$ and $\norm{\cdot}_{q}$ respectively to denote the inner product and two-norm of elements of the GNS Hilbert space}.

\subsection{Modular Theory}\label{modular}
Most of what follows in \S\ref{modular} and \S\ref{Commute} is taken from \cite{Shlyakhtenko,Hiai}. We need to have a convenient description of the commutant and centralizer of $\Gamma_q( \CH_\R, U_t)^{\prime\prime}$ $($which has been recorded in the case $q=0$ in \cite{Shlyakhtenko} and a similar collection of operators in the commutant has been identified in \cite{Hiai}$)$. Thus, we need to record some facts related to the modular theory of the $q$-quasi free state $\varphi_{q,U}$. The modular theory of $\Gamma_{q}(\CH_{\R}, U_{t})^{\prime\prime}$ associated to $\varphi_{q,U}$ is as follows. Let $J_{\varphi_{q,U}}$ and $\Delta_{\varphi_{q,U}}$ respectively denote the modular conjugation and modular operator associated to $\varphi_{q,U}$ and let $S_{\varphi_{q,U}}=J_{\varphi_{q,U}}\Delta_{\varphi_{q,U}}^{\frac{1}{2}}$. Then, for $n\in\N$, 
\begin{align}\label{modulartheory}
&J_{\varphi_{q,U}}(\xi_1 \otimes \cdots \otimes \xi_n) = A^{-1/2}\xi_n \otimes \cdots \otimes A^{-1/2}\xi_1, \text{ }\forall\text{ } \xi_{i}\in \mathcal{H}_{\mathbb{R}}\cap \mathfrak{D}(A^{-\half});\\
\nonumber&\Delta_{\varphi_{q,U}}(\xi_1 \otimes\cdots \otimes \xi_n ) =
A^{-1}\xi_1 \otimes \cdots\otimes A^{-1}\xi_n, \text{ }\forall\text{ } \xi_{i}\in \mathcal{H}_{\mathbb{R}}\cap \mathfrak{D}(A^{-1});\\
\nonumber&S_{\varphi_{q,U}}(\xi_1 \otimes \cdots \otimes \xi_n) =\xi_n \otimes \cdots \otimes \xi_1, \text{ }\forall\text{ } \xi_{i}\in \mathcal{H}_{\mathbb{R}}.
\end{align}
The modular automorphism group $(\sigma_{t}^{\varphi_{q,U}})$ of $\varphi_{q,U}$ is given by $\sigma_{-t}^{\varphi_{q,U}}=\text{Ad}(\CF(U_{t}))$, where $\CF(U_{t})=id\oplus \oplus_{n\geq 1} U_{t}^{\otimes_q n}$, for all $t\in \mathbb{R}$. In particular,
\begin{align}\label{modularaut}
\sigma^{\varphi_{q,U}}_{-t}(s_q(\xi)) = s_q(U_{t}\xi), \text{ for all } \xi \in \mathcal{H}_{\mathbb{R}}.
\end{align}
\subsection{Commutant}\label{Commute} Now we proceed to describe the commutant of $\Gamma_q( \CH_\R, U_t)^{\prime\prime}$. Consider the  set
\begin{align*}
\CH_\R' = \{ \xi \in \CH : \langle \xi, \eta \rangle_U \in \R \text{ for all }  \eta \in \CH_\R \}. 
\end{align*}
Note that $\overline{ \CH_\R' + i\CH_\R' } = \CH$ and $ \CH_\R' \cap i \CH^{\prime}_\R = \{ 0\}$. Let $ \zeta \in \mathfrak{D}(A^{-1/2})\cap \CH_\R$. Note that for all $\eta \in \CH_\R$, one has
\begin{align}\label{A-Inner}
\langle A^{-1/2} \zeta, \eta \rangle_U &=\langle \frac{2A^{-1/2}}{1+A^{-1}} \zeta, \eta \rangle_{\CH_\C}
=\langle \eta,  \CJ \frac{ 2A^{-1/2}}{1+A^{-1}} \zeta\rangle_{\CH_\C}\\
\nonumber&= \langle \eta,  \frac{ 2A^{1/2}}{1+A} \zeta \rangle_{\CH_\C}
= \langle \frac{ 2}{1+A^{-1}}\eta,  A^{-1/2} \zeta \rangle_{\CH_\C}\\
\nonumber& =\langle \eta,  A^{-1/2} \zeta \rangle_U.   
\end{align}
From Eq. \eqref{A-Inner}, it follows that 
\begin{equation}\label{prime}
 A^{-1/2} \zeta \in \CH_\R^\prime \text{ for all }\zeta \in \mathfrak{D}(A^{-\half})\cap \CH_\R.
\end{equation}
Also note that for $\eta,\xi\in \mathfrak{D}(A^{-1})\cap\CH_\R$, one has
\begin{align}\label{A-Inner-1}
\langle\eta,\xi \rangle_U &=\langle \frac{2}{1+A^{-1}} \eta, \xi \rangle_{\CH_\C}
=\langle \xi,  \CJ \frac{ 2}{1+A^{-1}} \eta\rangle_{\CH_\C}\\
\nonumber&= \langle \xi,  \frac{ 2}{1+A} \eta \rangle_{\CH_\C}=\langle \frac{ 2}{1+A^{-1}}\xi,  A^{-1} \eta\rangle_{\CH_\C}\\
\nonumber&=\langle \xi,  A^{-1} \eta \rangle_U=\langle A^{-\half}\xi,  A^{-\half} \eta \rangle_U \indent\text{ (as }\mathfrak{D}(A^{-1})\subseteq \mathfrak{D}(A^{-\half})).
\end{align}

Now for $\xi\in \CH$, define the right creation operator $r_q(\xi)$ on $\CF_q(\CH) $ by  
\begin{align}\label{Rightmult}
& r_q(\xi) \Omega = \xi, \\
\nonumber&r_q(\xi) (\xi_1 \otimes \cdots \otimes \xi_n) 
=  \xi_1 \otimes \cdots \otimes  \xi_n \otimes \xi, \text{ }\xi_{i}\in \CH, n\geq 1. 
\end{align}
Clearly, $r_q(\xi) = \jmath c_q(\xi) {\jmath}^{*}$, where $\jmath: \mathcal{F}_{q}(\CH) \rightarrow \mathcal{F}_{q}(\CH) $ is the unitary defined by 
\begin{align}\label{flip}
&\jmath ( \xi_1 \otimes \cdots \otimes \xi_n ) =  \xi_n \otimes \cdots \otimes \xi_1, \text{ where }\xi_{i}\in \CH \text{ for all }1\leq i\leq n, n\geq 1,\\
\nonumber&\jmath(\Omega)=\Omega.
\end{align}
Therefore, $r_q(\xi)$ is a bounded operator on $\CF_q(\CH) $ and its adjoint $r_q(\xi)^*$ is given by 
\begin{align}\label{Rightmultadj}
&r_q(\xi)^* \Omega = 0, \\
\nonumber&r_q(\xi)^* ( \xi_1 \otimes \cdots \otimes \xi_n) = \sum_{i = 1}^n q^{n-i}\langle  \xi , \xi_i \rangle_U \xi_1 \otimes \cdots \otimes \xi_{i-1} \otimes \xi_{i+1} \otimes \cdots \otimes \xi_n, \text{ }\xi_{i}\in \CH, n\geq 1. 
\end{align} 

Write $ d_q(\xi) = r_q(\xi) + r_q(\xi)^*$, $\xi\in\CH$. It is easy to observe that $ \{ d_q(\xi) : \xi \in \mathcal{H}_\R^{\prime} \} \subseteq \Gamma_q(\CH_\R, U_t)^{\prime}$. The following result establishes that the reverse inclusion is also true and its proof is similar to the one obtained in  \cite[Thm. 3.3]{Shlyakhtenko}.

\begin{Theorem}\label{commutant}
Suppose $ \xi \in \dom(A^{-1})\cap \CH_\R $. Then $J_{\varphi_{q,U}}s_q(\xi) J_{\varphi_{q,U}} = d_q(A^{-\fhalf} \xi)$.
Moreover, $\Gamma(\CH_{\R}, U_{t})^{\prime}=\{d_q(\xi): \xi\in \CH_\R^\prime\}^{\prime\prime}$.
\end{Theorem}

\begin{proof}
Fix $n\geq 1$ and let $ \eta_1, \eta_2, \cdots, \eta_n \in \dom(A^{-1})\cap\CH_\R$. Then from Eq. \eqref{modulartheory}, we have  
\begin{align*}
J_{\varphi_{q,U}}&s_q(\xi) ( \eta_1 \otimes \eta_2 \otimes \cdots \otimes \eta_n )\\ 
&= J_{\varphi_{q,U}}\left ( \sum_{i = 1}^n q^{(i-1)}\langle \xi, \eta_i \rangle_U \eta_1 \otimes \cdots \otimes \eta_{i-1} \otimes \eta_{i+1}  \cdots\otimes  \eta_n\right)
 + J_{\varphi_{q,U}} (\xi \otimes \eta_1 \otimes \cdots \otimes \eta_n)\\
&=\sum_{i = 1}^n q^{(i-1)} \langle \eta_i, \xi \rangle_{U} A^{-\half} \eta_n \otimes \cdots  \otimes A^{-\half}\eta_{i+1} \otimes A^{-\half} \eta_{i-1}\otimes \cdots \otimes A^{-\half}\eta_1\\
&\indent\indent\indent\indent+ A^{-\half} \eta_n \otimes \cdots \otimes A^{-\half}\eta_1 \otimes A^{-\half}\xi \text{ (since }\mathfrak{D}(A^{-1})\subseteq \mathfrak{D}(A^{-\half}))\\
&= \sum_{i = 1}^n q^{(i-1)}\langle \xi, A^{-1}\eta_i \rangle_{U} A^{-\half} \eta_n \otimes \cdots  \otimes A^{-\half}\eta_{i+1} \otimes A^{-\half} \eta_{i-1} \otimes  \cdots \otimes  A^{-\half}\eta_1 \\ 
&\indent\indent\indent\indent + A^{-\half} \eta_n \otimes \cdots \otimes A^{-\half}\eta_1 \otimes A^{-\half}\xi  \text{ (by Eq. }\eqref{A-Inner-1})\\
&=\sum_{i = 1}^n  q^{(i-1)} \langle A^{-\half} \xi, A^{-\half}\eta_i \rangle_{U} A^{-\half} \eta_n \otimes \cdots \otimes A^{-\half}\eta_{i+1} \otimes A^{-\half} \eta_{i-1} \otimes  \cdots \otimes  A^{-\half}\eta_1 \\
& \indent\indent\indent\indent+ A^{-\half} \eta_n \otimes \cdots \otimes A^{-\half}\eta_1 \otimes A^{-\half}\xi  \text{ (since }\mathfrak{D}(A^{-1})\subseteq \mathfrak{D}(A^{-\half}))\\
&= d_q(A^{-\half}\xi)J_{\varphi_{q,U}} ( \eta_1 \otimes \eta_2 \otimes \cdots \otimes \eta_n) \text{ (from Eq. \eqref{Rightmult} and Eq. \eqref{Rightmultadj})}. 
\end{align*}
It follows that $J_{\varphi_{q,U}}s_q(\xi)J_{\varphi_{q,U}} = d_q(A^{-\fhalf}\xi)$. 

Since $\Gamma_q(\CH_\R, U_t)^{\prime\prime}$ is in standard form in $\CF_q(\CH)$, so from the fundamental theorem of Tomita-Takesaki theory $\Gamma_q(\CH_\R, U_t)^{\prime} = J_{\varphi_{q,U}} \Gamma_q(\CH_\R, U_t)^{\prime\prime}J_{\varphi_{q,U}}$. Again from Eq. 
\eqref{prime}, one has $ A^{-\half}\xi \in \CH_\R^\prime$ for all $ \xi \in \dom(A^{-\fhalf})\cap \CH_\R$. By what we have proved so far, it follows that  $ \{J_{\varphi_{q,U}}s_{q}(\xi)J_{\varphi_{q,U}}:\xi\in \mathfrak{D}(A^{-1})\cap \CH_\R\}\subseteq  \{d_q(\xi): \xi\in \CH_\R^\prime\}^{\prime\prime}$. Note that from Eq. \eqref{Leftmult} it follows that, if $\CH_\R\ni\xi_{n}\rightarrow \xi\in \CH_\R$ in $\norm{\cdot}_{\CH_\C}$ $($equivalently in $\norm{\cdot}_{U})$, then
$s_q(\xi_n)\rightarrow s_q(\xi)$ in $\norm{\cdot}$ $($as $\norm{s_q(\zeta)}=\frac{2}{\sqrt{1-q}}\norm{\zeta}_{U}$ for all $\zeta\in\CH_\R)$. Consequently,  $\dom(A^{-1})\cap \CH_\R$ being dense in $\CH_\R$, it follows that $\Gamma_q(\CH_\R, U_t)^{\prime}\subseteq  \{d_q(\xi): \xi\in \CH_\R^\prime\}^{\prime\prime}$. 
Since the reverse inclusion is straight forward to check, the proof is complete. 
\end{proof}

\subsection{Notations and some technical facts} In this paper, we are interested in the factoriality of $\Gamma_q(\CH_\R, U_t)^{\prime\prime}$ and the orthogonal representation remain arbitrary but fixed. Thus, to reduce notation, we will write $M_q=\Gamma_q(\CH_\R, U_t)^{\prime\prime}$ and $\varphi=\varphi_{q,U}$. We will also denote $J_{\varphi_{q,U}}$ by $J$ and $\Delta_{\varphi_{q,U}}$ by $\Delta$. As $\Omega$ is separating for both $M_q$ and $M_q^{\prime}$, for $\zeta\in M_q\Omega$ and $\eta\in M_q^{\prime}\Omega$ there exist unique $x_{\zeta}\in M_q$ and $x^{\prime}_{\eta}\in M_q^{\prime}$ such that $\zeta =x_{\zeta}\Omega$ and $\eta=x_{\eta}^{\prime}\Omega$. In this case, we will write 
\begin{align}\label{sq}
s_q(\zeta)=x_{\zeta} \text{ and } d_q(\eta)=x_{\eta}^{\prime}. 
\end{align}
Thus, for example, as $\xi\in M_q\Omega$ for every $\xi\in \CH_\R$, so $s_q(\xi+i\eta)=s_q(\xi)+is_q(\eta)$ for all $\xi,\eta\in\CH_\R$.

\noindent {\bf Caution}: Note that $c_q(\xi)$ and $r_q(\xi)$ are bounded operators for all $\xi \in \CH$. 
Write 
\begin{align*}
\widetilde{s}_q(\xi) = c_q(\xi) + c_q(\xi)^* \text{ and } \widetilde{d}_q(\xi) = r_q(\xi) + r_q(\xi)^*, \text{ }\xi \in \CH.
\end{align*}
Note that if $\xi \in \CH_\R $, then  $\widetilde{s}_q(\xi)  = s_q(\xi)$, and if 
$\xi \in \CH_\R^\prime$ then $\widetilde{d}_q(\xi)= d_q(\xi)$. If $\xi = \xi_1 + i \xi_2 $ for $\xi_1, \xi_2 \in \CH_\R$ and $\xi_2 \neq 0 $, then note that 
$\widetilde{s}_q(\xi)\neq s_q(\xi) $.

Write $\mathcal{Z}(M_q)= M_q\cap M_q^{\prime}$. Let $M_q^{\varphi}$ denote the centralizer of $M_q$ associated to the state $\varphi$. For $\xi \in \CH_\R$, denote  $M_{\xi} =vN(s_q(\xi))$. Note that $M_\xi$ is abelian as $s_q(\xi)$ is self-adjoint. To understand the Hilbert space $\CF_{q}(\CH)$ as a bimodule over $M_{\xi}$, it will be convenient for us to work with appropriate choice of orthonormal basis of $\CH_\R$ with respect to $\langle\cdot,\cdot\rangle_{\CH_\C}$. We say that a vector $\xi\in \CH_\R$ is \textit{analytic}, if $s_{q}(\xi)$ is analytic for $(\sigma_t^{\varphi})$. 

\begin{Proposition}\label{analytic}
$\CH_\R$ has an orthonormal basis with respect to $\langle\cdot,\cdot\rangle_{\CH_\C}$ comprising of analytic vectors. Further, if $\xi_0\in\CH_\R$ be a unit vector such that $U_t\xi_0=\xi_0$ for all $t\in\R$, then such an orthonormal basis of $\CH_\R$ can be chosen so that it includes $\xi_0$. 
\end{Proposition}

\begin{proof}
Note that $U_t=A^{it}$ for all $t\in \R$. 
For $\zeta\in \CH_\R$ and $r>0$, let $\zeta_{r}=\sqrt{\frac{r}{\pi}}\int_{\R} e^{-rt^{2}}U_{t}\zeta dt$. It is well known that $\zeta_r\rightarrow \zeta$ in $\norm{\cdot}_{\CH_\C}$ $($equivalently in $\norm{\cdot}_{U}$ as the vectors involved are real$)$ as $r\rightarrow 0$. As $(U_{t})$ reduces $\CH_\R$ and $\zeta\in M_q\Omega$, so $\zeta_r\in M_q\Omega$ for all $r>0$. Fix $r>0$. Consider $s_{q}(\zeta_{r})\in M_q$ $($as defined in Eq. \eqref{sq}$)$. Then, by Eq. \eqref{modularaut} it follows that
\begin{align}\label{analytic_extension}
\sigma_{s}^{\varphi}(s_{q}(\zeta_{r}))=s_{q}(\sqrt{\frac{r}{\pi}}\int_{\R} e^{-r(t+s)^{2}}U_{t}\zeta dt), \text{ }s\in \R.
\end{align}
Note that $f_{\zeta_r}(z)=\sqrt{\frac{r}{\pi}}\int_{\R} e^{-r(t+z)^{2}}U_{t}\zeta dt \in \CH_\C$ for all $z\in \C$. Thus, $s_q(f_{\zeta_r}(z))$ is defined by Eq. \eqref{sq} and belongs to $M_q$. It is easy to see that $s_q(f_{\zeta_r}(\cdot)): \C\rightarrow M_q$ is an analytic extension of $\R\ni s\mapsto \sigma_{s}^{\varphi}(s_{q}(\zeta_{r}))$. Thus, $\zeta_{r}$ is analytic. 

Let $\mathfrak{D}_{0}=\text{span}_\R\{\zeta_{r}: r>0, \zeta\in\CH_\R\}$. Note that $\mathfrak{D}_{0}$ $($consisting of analytic vectors$)$ is dense in $(\CH_\R,\langle\cdot,\cdot\rangle_{\CH_\C})$. Finally, use the fact that any dense subspace of a separable $($real$)$ Hilbert space has an orthonormal basis consisting of elements from the dense subspace. The rest is clear. We omit the details. 
\end{proof}

\begin{Remark}\label{analytic_extension_basis}
Note that from Eq. \eqref{sq} and Eq. \eqref{analytic_extension}, it follows that $\sigma^{\varphi}_z(s_q(\zeta_r))=s_q(f_{\zeta_r}(z))$ for all $z\in \C$ and $r>0$ $($where $\zeta_r\in\mathfrak{D}_{0}$ as in the proof of Prop. \ref{analytic}$)$. From Eq. \eqref{modulartheory}, it follows that 
\begin{align*}
A^{-\half}\zeta_r=\sigma^{\varphi}_{-\frac{i}{2}}(s_q(\zeta_r))\Omega=s_q\bigl(f_{\zeta_r}(-\frac{i}{2})\bigr)\Omega = \sqrt{\frac{r}{\pi}}\int_{\R} e^{-r(t-\frac{i}{2})^{2}}U_{t}\zeta dt.
\end{align*} 
Decomposing into real and imaginary parts and arguing as in the proof of Prop. \ref{analytic}, it is easy to check that $A^{-\half}\zeta_r$ is analytic for all $r>0$ as well. 
\end{Remark}

Note that $\langle\cdot,\cdot\rangle_{\CH_\C}$ is clearly not the inner product in the GNS space but is connected to the latter. In fact, the following observation will be crucial all throughout the paper and tries to find instances when `orthogonality' with respect to one inner product entails `certain orthogonality' with respect to the other. The following is an analogue of the fact that a word in the free group $\mathbb{F}_2=\langle a,b\rangle$  is orthogonal to the generator masa $vN(a)$ with respect to the trace if and only if one letter in the word is $b$ or $b^{-1}$. 

\begin{Lemma}\label{Orthogonality}
Let $\xi_0\in\CH_\R$ be a unit vector such that $U_t \xi_0 =\xi_0$ for all $t$. Then the following hold.
\begin{enumerate}
\item For $\eta\in \CH_\R+ i\CH_\R$ one has
\begin{align*}
\langle \xi_0, \eta\rangle_U=\langle \xi_0,\eta\rangle_{\CH_\C}.
\end{align*}
\item Let $\xi_1,\cdots,\xi_n\in\CH_\R$ be non zero vectors. If $k\geq 1$, then 
\begin{align*}
\langle \xi_0^{\otimes k}, \xi_1\otimes\cdots\otimes\xi_n\rangle_{q}=0,
\end{align*}
if and only if $n\neq k$ or $\langle \xi_0,\xi_i\rangle_{\CH_\C}=0$ for at least one $i$.
\item Let $\xi_1,\cdots,\xi_n\in\CH_\R\cap \mathfrak{D}(A^{-\frac{1}{2}})$ be non zero vectors. If $k\geq 1$, then 
\begin{align*}
\langle \xi_0^{\otimes k}, A^{-\frac{1}{2}}\xi_1\otimes\cdots\otimes A^{-\frac{1}{2}}\xi_n\rangle_{q}=0,
\end{align*}
if and only if $n\neq k$ or $\langle \xi_0,\xi_i\rangle_{\CH_\C}=0$ for at least one $i$.
\end{enumerate}
\end{Lemma}

\begin{proof}
\noindent $(1)$. Note that  $\frac{2}{1+A^{-1}}\xi_0=\xi_0$. Thus, the result follows from the definition of $\langle\cdot,\cdot\rangle_{U}$.

\noindent$(2)$. Note that 
\begin{align*}
\langle \xi_0^{\otimes k}, \xi_1\otimes\cdots\otimes\xi_n\rangle_{q} &=\delta_{n,k} \langle \xi_0^{\otimes n}, \xi_1\otimes\cdots\otimes\xi_n\rangle_{q}\\
&=\delta_{n,k}\sum_{\pi\in S_n} q^{i(\pi)} \prod_{j=1}^n\langle \xi_0, \xi_{\pi(j)}\rangle_{U} \text{ (by Eq. \eqref{qFock})}\\
&=\delta_{n,k}\sum_{\pi\in S_n} q^{i(\pi)} \prod_{j=1}^n\langle \frac{2}{1+A^{-1}}\xi_0, \xi_{\pi(j)}\rangle_{\CH_\C} \\ 
&=\delta_{n,k}\sum_{\pi\in S_n} q^{i(\pi)} \prod_{j=1}^n\langle \xi_0, \xi_{\pi(j)}\rangle_{\CH_\C} \\
&=\delta_{n,k}\prod_{j=1}^n\langle \xi_0, \xi_{j}\rangle_{\CH_\C}\sum_{\pi\in S_n} q^{i(\pi)}  \\
&=\delta_{n,k}\prod_{j=1}^n\langle \xi_0, \xi_{j}\rangle_{\CH_\C}[n]_q!.
\end{align*}
The rest is immediate. 

\noindent$(3)$. First note that as $\xi_i\in\CH_\R$, so $A^{-\frac{1}{2}}\xi_i\in \CH_\C$. Observe that 
\begin{align*}
\langle \xi_0^{\otimes k}, A^{-\frac{1}{2}}\xi_1\otimes\cdots\otimes A^{-\frac{1}{2}}\xi_n\rangle_{q} &=\delta_{n,k} \langle \xi_0^{\otimes n}, A^{-\frac{1}{2}}\xi_1\otimes\cdots\otimes A^{-\frac{1}{2}}\xi_n\rangle_{q}\\
&=\delta_{n,k}\sum_{\pi\in S_n} q^{i(\pi)} \prod_{j=1}^n\langle \xi_0, A^{-\frac{1}{2}}\xi_{\pi(j)}\rangle_{U} \text{ (by Eq. \eqref{qFock})}\\ 
&=\delta_{n,k}\sum_{\pi\in S_n} q^{i(\pi)} \prod_{j=1}^n\langle \xi_0, \xi_{\pi(j)}\rangle_{U}\\
&=\langle \xi_0^{\otimes k}, \xi_1\otimes\cdots\otimes\xi_n\rangle_{q}.
\end{align*}
Thus, the result follows from $(2)$ above.
\end{proof}


\section{Centralizer}\label{SectionCentralizer}

A convenient description of the centralizer $M_q^\varphi=\{x\in M_q: \sigma_t^\varphi(x)=x \text{ }\forall \text{ }t\in \R\}$ is a component we need to decide the factoriality and type of $M_q$. In this section, we borrow ideas from Thm. 2.2 of  \cite{Hiai} and show that the centralizer of $M_q$ depends on the almost periodic part of the orthogonal representation $(U_t)$. We need some intermediate results.

\begin{Lemma}\label{VectorinMq}
The following hold.
\begin{enumerate}
\item The vector $\xi_{1}\otimes\cdots\otimes \xi_{n}\in M_q\Omega $ for any $\xi_i\in\CH_\R$, $1\leq i\leq n$ and $n\in \mathbb{N}$.
\item The vector $\xi_{1}\otimes\cdots\otimes \xi_{n}\in M_q^{\prime}\Omega $ for any $\xi_i\in\mathfrak{D}(A^{-\half})\cap\CH_\R$, $1\leq i\leq n$ and $n\in \mathbb{N}$.
\end{enumerate}
\end{Lemma}

\begin{proof}
In both cases, the proof proceeds by induction.

\noindent$(1)$ Let $n=1$. Then by definition of $M_q$ it follows that $\xi=s_q(\xi)\Omega\in M_q\Omega$ for all $\xi\in\CH_\R$. Now suppose that $\xi_{1}\otimes\cdots\otimes \xi_{t}\in M_q\Omega$ for all $\xi_j\in\CH_\R$, $1\leq j\leq t$ and for all $1\leq t\leq n$. Let $\xi_{n+1}\in\CH_\R$. Then from Eq. \eqref{Leftmult} we have, 
\begin{align*}
\xi_{1}\otimes\cdots\otimes \xi_{n}\otimes\xi_{n+1}=s_q(\xi_{1})&s_q(\xi_{2}\otimes\cdots\otimes\xi_{n+1})\Omega&\\
&-\sum_{l=2}^{n+1}{q^{l-2}}
\langle \xi_{1},\xi_{l}\rangle_U \xi_{2} \otimes \cdots \otimes \xi_{l-1} \otimes \xi_{l+1} \otimes \cdots \otimes \xi_{n+1}.
\end{align*}
But the right hand side of the above expression lies in $M_q\Omega$ by the induction hypothesis. Thus, $\xi_{1}\otimes\cdots\otimes \xi_n\in M_q\Omega$ for $\xi_i\in \CH_\R$, $1\leq i\leq n$ and for all $n\in \mathbb{N}$.

\noindent$(2)$ Let $\xi\in\mathfrak{D}(A^{-\half})\cap\CH_\R$. By Eq. \eqref{modulartheory}, it follows that $J(\CH_\R)\subseteq \CH_\C$. Thus, write $J\xi=\eta_1+i\eta_2$ with $\eta_1,\eta_2\in\CH_\R$. Then $s_q(\eta_j)\in M_q$ for $j=1,2$, thus $J\xi = \left(s_q(\eta_1)+is_q(\eta_2)\right)\Omega\in M_q\Omega$. Note that $Js_q(J\xi)J\Omega=\xi$. Consequently, $\xi\in M_q^{\prime}\Omega$ by the fundamental theorem of Tomita-Takesaki theory. Like before, assume that $\xi_{1}\otimes\cdots\otimes \xi_{t}\in M_q^{\prime}\Omega$ for all $\xi_j\in\CH_\R\cap \mathfrak{D}(A^{-\frac{1}{2}})$, $1\leq j\leq t$ and for all $1\leq t\leq n$. 

Fix $\xi_{n+1}\in\mathfrak{D}(A^{-\half})\cap \CH_\R$ and let $J\xi_{n+1} = A^{-\half} \xi_{n+1} = \eta^1_{n+1} +i\eta^2_{n+1}$ with $\eta^1_{n+1},\eta^2_{n+1}\in\CH_\R$ $($see Eq. \eqref{modulartheory}$)$. Then for $\xi_i\in \mathfrak{D}(A^{-\half})\cap \CH_\R$ for all $1\leq i\leq n$, from Eq. \eqref{Leftmult}, Eq. \eqref{modulartheory}, and the fact that  $J^2=1$, it follows that
\begin{align*} 
Js_q(J\xi_{n+1})J &d_q(\xi_{1}\otimes\cdots\otimes\xi_{n})\Omega\\
&=Js_q(J\xi_{n+1})J(\xi_{1}\otimes\cdots\otimes\xi_{n})\\
&=J \left((c_q(\eta^1_{n+1}) + ic_q(\eta^2_{n+1}))  (A^{-\half}\xi_{n}\otimes\cdots\otimes A^{-\half}\xi_{1})\right) \\
& \indent\indent+ J \left((c_q(\eta^1_{n+1})^* + ic_q(\eta^2_{n+1})^*)  (A^{-\half}\xi_{n}\otimes\cdots\otimes A^{-\half}\xi_{1})\right) \\
&=J \left((\eta^1_{n+1} + i\eta^2_{n+1})\otimes  (A^{-\half}\xi_{n}\otimes\cdots\otimes A^{-\half}\xi_{1})
\right) \\
&\indent\indent+ J \left((c_q(\eta^1_{n+1})^* + ic_q(\eta^2_{n+1})^*)  (A^{-\half}\xi_{n}\otimes\cdots\otimes A^{-\half}\xi_{1})\right) \\
&=J \left(A^{-\half}\xi_{n+1 } \otimes  A^{-\half}\xi_{n}\otimes\cdots\otimes A^{-\half}\xi_{1}\right) \\
&\indent\indent+ J \left((c_q(\eta^1_{n+1})^* + ic_q(\eta^2_{n+1})^*)  (A^{-\half}\xi_{n}\otimes\cdots\otimes A^{-\half}\xi_{1})\right) \\
&=\xi_{1}\otimes\cdots\otimes \xi_{n}\otimes\xi_{n+1} 
+J \left((c_q(\eta^1_{n+1})^* + ic_q(\eta^2_{n+1})^*)  (A^{-\half}\xi_{n}\otimes\cdots\otimes A^{-\half}\xi_{1})
\right).
\end{align*}
Using the induction hypothesis, Eq. \eqref{Leftmult} and decomposing vectors in $\CH_\C$ into real and imaginary parts, it is straightforward to check that 
\begin{align*}
J \left((c_q(\eta^1_{n+1})^* + ic_q(\eta^2_{n+1})^*)(A^{-\half}\xi_{n}\otimes\cdots\otimes A^{-\half}\xi_{1})\right) \in M_q^{\prime} \Omega. 
\end{align*}
Hence, $\xi_{1}\otimes\cdots\otimes \xi_{n}\otimes\xi_{n+1}\in  M_q^{\prime}\Omega$. Now use induction to complete the proof.
\end{proof}

In the next Lemma, we make use of Lemma \ref{VectorinMq} to show how certain operators in $M_q$ act on simple tensors. 

\begin{Lemma}\label{Wickformula}
Let $\xi, \xi_i \in \CH_\R $ for $1 \leq i \leq n$ be such that $\langle 
\xi_i, \xi \rangle_U = 0$ for $ 1\leq i \leq n $. Then, 
\begin{align*}
s_q(\xi_1\otimes \cdots \otimes \xi_n)(\xi^{\otimes k})= 
\xi_1 \otimes \cdots \otimes \xi_n \otimes \xi^{\otimes k}, \text{ for all }k\geq 0. 	
\end{align*}
\end{Lemma}

\begin{proof} 
Note that by Lemma \ref{VectorinMq}, it follows that $s_q(\xi_1\otimes \cdots \otimes \xi_n)\in M_q$. The result is clearly true for $k=0$ by definition $($see Eq. \eqref{sq}$)$. We will only prove this result for $k=1$. For $k\geq 2$, the argument is similar. 

We use induction. Let $n=1$, then note that,  
\begin{align*}
s_q(\xi_1)\xi=\xi_1\otimes \xi+\langle\xi_1, \xi \rangle_U \Omega =\xi_1 \otimes \xi, \text{ by Eq. \eqref{Leftmult}}.  
\end{align*}
Now suppose that the result is true for all $1\leq m\leq n$.  Let $ \xi_{n+1} \in \CH_\R$ be such that $\langle 
\xi_{n+1}, \xi \rangle_U = 0$.  Then, from Eq. \eqref{Leftmult} and the proof of Lemma \ref{VectorinMq}, we have
\begin{align*}
s_q(\xi_{1}\otimes\cdots\otimes \xi_{n}\otimes\xi_{n+1})&=s_q(\xi_{1})s_q(\xi_{2}\otimes\cdots\otimes\xi_{n+1})\\
&\indent -\sum_{l= 2}^{n+1}{q^{l-2}}
\langle \xi_{1},\xi_{l}\rangle_U s_q(\xi_{2} \otimes \cdots \otimes \xi_{l-1} \otimes \xi_{l+1} \otimes \cdots \otimes \xi_{n+1}).
\end{align*}
Consequently, by using the induction hypothesis, one has 
\begin{align*}
&s_q(\xi_{1}\otimes\cdots\otimes \xi_{n}\otimes\xi_{n+1}) \xi \\
=& 	s_q(\xi_{1})s_q(\xi_{2}\otimes\cdots\otimes\xi_{n+1}) \xi\\
&\indent-\sum_{l=2}^{n+1}{q^{l-2}}
\langle \xi_{1},\xi_{l}\rangle_U s_q(\xi_{2} \otimes \cdots \otimes \xi_{l-1} \otimes \xi_{l+1} \otimes \cdots \otimes \xi_{n+1})\xi\\
=& 	s_q(\xi_{1})(\xi_{2}\otimes\cdots\otimes\xi_{n+1}\otimes \xi)\\
&\indent-\sum_{l=2}^{n+1}{q^{l-2}}
\langle \xi_{1},\xi_{l}\rangle_U (\xi_{2} \otimes \cdots \otimes \xi_{l-1} \otimes \xi_{l+1} \otimes \cdots \otimes \xi_{n+1} \otimes \xi)\\
=& \xi_{1}\otimes\cdots\otimes \xi_{n}\otimes\xi_{n+1}\otimes \xi, \text{ by Eq. \eqref{Leftmult}}.
\end{align*}
This completes the proof. 
\end{proof}

Since $t\mapsto U_t$, $t\in \R$, is a strongly continuous orthogonal representation of $\R$ on the real Hilbert space $\CH_\R$, so there is a unique decomposition $($c.f. \cite{Shlyakhtenko}$)$,
\begin{align}\label{Representation}
( \CH_\R, U_t)  =  
\left(\bigoplus_{j = 1}^{N_1} (\R, \text{id} )\right) \oplus \left( \bigoplus_{k=1}^{N_2}(\CH_\R(k), U_t(k) ) \right) \oplus (\widetilde{\CH}_\R, \widetilde{U}_t),
\end{align}
where $0\leq N_1,N_2\leq \aleph_0$, 
\begin{align*}
\CH_\R(k) = \R^2, \quad U_t(k)  = \left( \begin{matrix} \cos( t\log \lambda_k)& - \sin( t\log\lambda_k) \\
\sin( t\log\lambda_k)& \cos( t\log\lambda_k)  \end{matrix} \right), \text{ }\lambda_k > 1,
\end{align*}
and $(\widetilde{\CH}_\R, \widetilde{U}_t)$ corresponds to the weakly mixing component of the orthogonal representation; thus $\widetilde{\CH}_\R$ is either $0$ or infinite dimensional.

If $N_1\neq 0$, let $e_j=0\oplus\cdots\oplus 0\oplus 1\oplus 0\oplus\cdots\oplus 0\in \bigoplus_{j=1}^{N_1}\R$, where $1$ appears at the $j$-th place for $1\leq j\leq N_1$. Similarly, if $N_2\neq 0$, let $f_k^1=0\oplus\cdots\oplus 0\oplus\left( \begin{matrix}1\\0 \end{matrix}\right)\oplus 0\oplus\cdots \oplus 0\in\bigoplus_{k=1}^{N_2}\CH_\R(k)$ and $f_k^2 =0\oplus\cdots\oplus 0\oplus\left( \begin{matrix}0\\1\end{matrix}\right)\oplus 0\oplus\cdots \oplus 0\in \bigoplus_{k=1}^{N_2}\CH_\R(k)$ be vectors with non zero entries in the $k$-th position for $1\leq k\leq N_2$. Denote 
\begin{align*}
e^1_k= \frac{\sqrt{\lambda_k +1}}{2}(f_k^1+if_k^2) \text{ and }e^2_k=\frac{\sqrt{{\lambda}^{-1}_k+1}}{2}(f_k^1-if_k^2),
\end{align*}
thus $e_k^1,e_k^2 \in\CH_\R(k)+i\CH_\R(k)$ are orthonormal basis of $(\CH_\R(k)+i\CH_\R(k),\langle\cdot,\cdot\rangle_{U})$ for $1\leq k\leq N_2$. Fix $1\leq k\leq N_2$. The analytic generator $A(k)$ of $(U_t(k))$ is given by 
\begin{align*}
A(k)=\frac{1}{2}\left( \begin{matrix} \lambda_k + \frac{1}{\lambda_k} & i(\lambda_k - \frac{1}{\lambda_k})\\
-i(\lambda_k -\frac{1}{ \lambda_k}) &\lambda_k + \frac{1}{\lambda_k}
\end{matrix}\right).
\end{align*}
Moreover, 
\begin{align*}
A(k) e_k^1 = \frac{1}{\lambda_k} e_k^1  \text{ and }  A(k) e_k^2 = \lambda_k e_k^2.
\end{align*}
Write $\CS=\{e_j: 1\leq j\leq N_1\} \cup \{e_k^1,e_k^2: 1\leq k\leq N_2\}$ if $N_1\neq 0$ or $N_2\neq 0$, else set $S=\{0\}$. If $S\neq \{0\}$, then $S$ is an orthogonal set in $(\CH_\C, \langle \cdot, \cdot\rangle_U)$ and the space of eigen vectors of the analytic generator $A$ of $(U_t)$ is contained in $\text{span }\CS$. In the event $\CS\neq \{0\}$, 
rename the elements of the set $\CS$ as $\zeta_1,\zeta_2,\cdots$, i.e., $\CS=\{\zeta_i:1\leq i\leq N_1+2N_2\}$, whence $A\zeta_l=\beta_l\zeta_l$ with $\beta_l\in\mathcal{E}_A$ for all $l$, where $\mathcal{E}_A=\{1\}\cup \{\lambda_k:1\leq k\leq N_2\}\cup \{\frac{1}{\lambda_k}:1\leq k\leq N_2\}$. It is to be understood that when $N_1=\infty$ $($resp. $N_2=\infty)$, the constraints $j\leq N_1$ and $i\leq N_1+2N_2$ $($resp. $k\leq N_2$ and $i\leq N_1+2N_2)$ $($in defining $\CS$ and $\mathcal{E}_A)$ is replaced by $j<N_1$ and $i<N_1+2N_2$ $($resp.  $k<N_2$ and $i<N_1+2N_2)$.

The following result must be known to experts but we lag a reference, so we prove it for the sake of convenience.

\begin{Proposition}\label{Eigenvalueinamplification}
Let $(\rho,\CH)$ be a strongly continuous unitary representation of a separable locally compact abelian group $G$ on a Hilbert space $\CH$. For $n\geq 1$ and $q\in (-1,1)$, let $\rho^{\otimes_q n}$ be the $n$-fold amplification of $\rho$ on $\CH^{\otimes_q n}$ defined by 
\begin{align*}
\rho^{\otimes_q n}(g)(\xi_1\otimes\cdots\otimes \xi_n)=\rho(g)\xi_1\otimes\cdots\otimes\rho(g)\xi_n, \text{ }g\in G, \xi_i\in \CH \text{ for }1\leq i\leq n.
\end{align*}
Then $(\rho^{\otimes_q n},\CH^{\otimes_q n})$ is a strongly continuous unitary representation of $G$. Let 
$\eta\in \CH^{\otimes_q n}$ be an eigen vector of $\rho^{\otimes_q n}$ with associated character $\chi\in \widehat{G}$. Let 
\begin{align*}
\mathfrak{e}_{\chi}=\Big\{\xi_1\otimes\cdots\otimes\xi_n: \xi_i\in\CH, \exists\text{ } &\chi_i\in\widehat{G} \text{ such that }\\
&\rho(\cdot)\xi_i=\chi_{i}(\cdot)\xi_i, 1\leq i\leq n, \prod_{i=1}^{n}\chi_i=\chi\Big\}.
\end{align*}
Then, $\eta\in \overline{\text{span }\mathfrak{e}_{\chi}}$.
\end{Proposition}

\begin{proof}
First of all, note that Eq. \eqref{qFock} forces that $\rho^{\otimes_q n}$ is a strongly continuous unitary representation of $G$. Note that by Lemma \ref{Equivalent}, the operator $S: (\CH^{\otimes_q n},\norm{\cdot}_q)\rightarrow (\CH^{\otimes_0 n},\norm{\cdot}_0)$ defined by $S(\xi_1\otimes\cdots\otimes \xi_n)=\xi_1\otimes\cdots\otimes \xi_n$, for all $\xi_i\in \CH$, $1\leq i\leq n$, is bounded and invertible. Moreover, $S^{-1}\rho^{\otimes_0 n}(\cdot)S=\rho^{\otimes_q n}(\cdot)$. Consequently, the spectral properties of $\rho^{\otimes_q n}$ and $\rho^{\otimes_0 n}$ are identical. Also note that $\rho^{\otimes_0 n}$ is the usual tensor product representation on the usual tensor product of Hilbert spaces. 

The result now follows from the following fact. If $U_j\in \textbf{B}(\mathcal{K})$ is unitary for $1\leq j\leq n$, then $\lambda\in S^1$ is an eigen value of $U_1\otimes \cdots \otimes U_n$ if and only if there exist $\lambda_j\in S^1$ and unit vectors $\xi_j\in\mathcal{K}$ such that $U_j\xi_j=\lambda_j\xi_j$ for $1\leq j\leq n$ and $\lambda =\prod_{j=1}^n \lambda_j$. The rest is obvious, we omit the details.
\end{proof}

\begin{Theorem}\label{CentraliserDescribe}
Let 
\begin{align*}
\CW_0 =\begin{cases}& \{\zeta_{i_1} \otimes \cdots \otimes \zeta_{i_n}:\zeta_{i_j} \in \CS, 1\leq i_j\leq  N_1+2N_2, \prod_{j=1}^n\beta_{i_j}= 1, n\in \N \},\\
&\indent\indent\indent\indent\indent\indent\indent\indent\indent \text{ if }\max(N_1,N_2)<\infty;\\
&\{\zeta_{i_1} \otimes \cdots \otimes \zeta_{i_n}:\zeta_{i_j} \in \CS, 1\leq i_j <N_1+2N_2, \prod_{j=1}^n\beta_{i_j}= 1, n\in \N \},\\
&\indent\indent\indent\indent\indent\indent\indent\indent\indent \text{ if }\max(N_1,N_2)=\infty.\\
\end{cases}
\end{align*} 
Let $\CW=\C\Omega\oplus \overline{\text{span }\CW_0}^{\norm{\cdot}_q}$. Then, $M_q^\varphi\Omega=\CW\cap M_q\Omega$.
\end{Theorem}

\begin{proof}
Decomposing vectors in $\CS$ into real and imaginary parts and using Lemma \ref{VectorinMq}, it follows that $\CW_0\subseteq M_q\Omega$. Fix $n\in \N$ and let $1\leq i_1,\cdots, i_n\leq N_1+2N_2$ or $1\leq i_1,\cdots, i_n< N_1+2N_2$ $($as the case may be$)$, be such that $\beta_{i_1}\cdots \beta_{i_n}=1$. Pick $\zeta_{i_j}\in \CS$ for $1\leq j\leq n$. Consider $x=s_q(\zeta_{i_1} \otimes \cdots\otimes\zeta_{i_n} )\in M_q$. As $\sigma_{-t}^{\varphi}=\text{Ad}(\CF(U_t))$ $($see Eq. \eqref{modulartheory}, \eqref{modularaut}$)$, so 
\begin{align*}
\sigma^\varphi_{-t} (x)\Omega 
&= \CF(U_t) x \CF(U_t)^* \Omega 
= \CF(U_t) x \Omega 
=\CF(U_t)(\zeta_{i_1}\otimes\cdots\otimes\zeta_{i_n})\\
&= U_t\zeta_{i_1}\otimes \cdots \otimes U_t\zeta_{i_n}\\
&= (\beta_{i_1}\cdots \beta_{i_n})^{it}(\zeta_{i_1}\otimes\cdots\otimes\zeta_{i_n}), \text{ (since } U_t = A^{it})\\
&= s_q(\zeta_{i_1}\otimes\cdots\otimes\zeta_{i_n} )\Omega\\
&=x\Omega, \text{ for all }t\in\R.
\end{align*}
Consequently, $x=s_q(\zeta_{i_1}\otimes\cdots\otimes\zeta_{i_n}) \in M^\varphi_q$. Therefore, conclude that $\CW \cap M_q\Omega \subseteq M_q^\varphi \Omega$.

For the reverse inclusion, let $y \in M^\varphi_q$ and write $y\Omega=\sum_{n = 0}^\infty\eta_n$, where ${\eta}_n\in \CH^{\otimes_q n}$ for all $n\geq 0$ and the series converges in $\norm{\cdot}_q$. It is enough to show that $\eta_n \in \CW$ for all $n\geq 0$. Again, note that 
\begin{align*}
\sum_{n = 0}^\infty\eta_n&=\sigma^\varphi_{-t} (y)\Omega=\CF(U_t) y\CF(U_t)^* \Omega\\
& = \CF(U_t) y\Omega \\
&=\CF(U_t) \sum_{n = 0}^\infty \eta_n
=\sum_{n = 0}^\infty\CF(U_t)  \eta_n, \text{ for all }t\in\R.
\end{align*}
Since $\CF(U_t) {\CH}^{\otimes_q n }={\CH}^{\otimes_q n }$ for all $n\geq 0$ and for all $t\in \R$, so we have $\CF(U_t)\eta_n =\eta_n$ for all $n$ and for all $t\in\R$. Fix $n\geq 1$ such that $\eta_n\neq 0$. Therefore, by Prop. \ref{Eigenvector} and Prop. \ref{Eigenvalueinamplification}, it follows that there exist $\zeta_{k,l}^{(n)}\in \CS$ and $\beta_{k,l}^{(n)}\in \mathcal{E}_A$ with $A\zeta_{k,l}^{(n)}=\beta_{k,l}^{(n)}\zeta_{k,l}^{(n)}$ for $1\leq k\leq n$ and scalars $c_{n,l}$, $l\in\N$, such that $\eta_n =\sum_{l}c_{n,l}(\zeta_{1,l}^{(n)}\otimes\cdots\otimes\zeta_{n,l}^{(n)})$ and $\prod_{k=1}^{n}\beta_{k,l}^{(n)}=1$ for all $l$; the series above converges in $\norm{\cdot}_{q}$. Consequently, $\eta_n\in \mathcal{W}$ for all $n\geq 0$ and the proof is complete.
\end{proof}

\begin{Remark}\label{CentralizerRemark}
\noindent$(i)$ Suppose $N_1=1,N_2=0$ and $\widetilde{H}_{\R}\neq 0$ in Eq. \eqref{Representation}. Then by Thm. \ref{CentraliserDescribe} we have $M_q^{\varphi}=M_{\xi}$, where $0\neq\xi\in \CH_\R$ is such that 
$U_t\xi=\xi$ for all $t\in\R$. From Thm. \ref{masa} $($proved later$)$, $M_\xi$ is a masa in $M_q$, so $(M_q^{\varphi})^{\prime}\cap M_q=M_q^{\varphi}$. Thus, the conclusion of Thm. 3.2 of \cite{Hiai} is not true in general.

\noindent$(ii)$ If $\mathcal{E}_A=\{1\}$, then $M_q^{\varphi}$ is isomorphic to the $q$-Gaussian von Neumann algebra of Bo$\overset{.}{\text{z}}$ejko and Speicher \cite{BS,BKS}. Thus, in this case, if $1$ is eigen value of multiplicity more than or equal to $2$, then $M_q^{\varphi}$ is a factor by \cite{ER} $($compare Cor. \ref{Trivialrelcommutant}$)$.
\end{Remark}


\section{Generator Algebras $M_\xi$}\label{GenMasas}
  
In this section, we investigate the von Neumann subalgebras $M_{\xi}$ for $\xi\in \CH_\R$, and record some of their properties. This is a preparatory section and the aforesaid subalgebras play major role in deciding the factoriality of $M_q$. 

In the case when $q =0$, $t\mapsto U_t$ is the identity representation of $\R$ and $dim(\CH_\R) \geq 2$, it is well known that $M_0=\Gamma_0(\CH_\R, id_t) \cong L\mathbb{F}_{dim(\CH_\R)}$  (see \cite{DVN}). In that case, for all $0\neq\xi \in \CH_\R$, the algebra $M_{\xi}$ is maximal injective $($see \cite{Po83}$)$, strongly mixing masa, for which the orthocomplement of the associated Jones' projection regarded as a $M_{\xi}$-bimodule is infinite direct sum of coarse bimodules $($see \cite{Jan-Fang-Kunal}, \cite{DSS06}$)$. Moreover, if $\xi_{1},\xi_{2}\in \CH_\R$ be non zero elements such that $\langle \xi_{1},\xi_{2}\rangle_{\CH_\C}=0$, then $M_{\xi_1}$ and $M_{\xi_2}$ are free and outer conjugate \cite{DVN}.

Note that if $0\neq\xi\in \CH_\R$ and $U_t\xi=\xi$ for all $t\in \R$, then $s_q(\xi) \in M_q^{\varphi}$ $($from Eq. \eqref{modularaut}$)$. So $J\xi= J s_q(\xi)\Omega = s_q(\xi)^{*}\Omega = s_q(\xi)\Omega = \xi$. 

By Eq. $(1.2)$ of \cite{Hiai}, for $\xi\in \CH_\R$ with $\norm{\xi}_{U}=1$, the moments of the operator $s_q(\xi)$ with respect to the $q$-quasi free state $\varphi(\cdot) = \lan \Omega, \cdot \Omega \ran_q$ is given by 
\begin{equation*}
\varphi(s_q(\xi)^n)=
   \begin{cases}
     0, & \text{if } n \text{ is odd}, \\
     \sum_{\CV = \{ \pi(r), \kappa(r) \}_{ 1 \leq r \leq \frac{n}{2} } }  q^{c(\CV)}, & \text{if }n\text{ is even},
  \end{cases}
\end{equation*}
where the summation is taken over all pair partitions $\CV = \{ \pi(r), \kappa(r) \}_{ 1 \leq r \leq \frac{n}{2} }$ of $\{1, 2, \cdots, n\}$ with $\pi(r)<\kappa(r) $ and $c(\CV)$ is the number of crossings of $\CV$, i.e, 
\begin{align*}
c(\CV) = \# \{ (r, s): \pi(r)< \pi(s) < \kappa(r) < \kappa(s) \}.
\end{align*}
So, it  follows that for $\xi\in\CH_\R$ with ${\norm{\xi}}_U = 1 $, the distribution of the single $q$-Gaussian $s_q(\xi)$ does not depend on the group $(U_t )$. In the tracial case, and thus in all cases, this distribution obeys the semicircular law $\nu_q$ which is absolutely continuous with respect to the uniform measure  supported on the interval $[-\frac{2}{\sqrt{ 1-q}}, \frac{2}{\sqrt{ 1-q}}]$. The associated orthogonal polynomials are $q$-Hermite polynomials $H_{n}^{q}$, $ n\geq 0$. For the density function of $\nu_q$ and the recurrence relations defining the $q$-Hermite polynomials, we refer the reader to Defn. $1.9$ and Thm. $1.10$ of  \cite{BKS} $($also see \cite{Nou-Araki}, \cite{DVN}$)$. Hence, $M_{\xi}\cong L^\infty([-\frac{2}{\sqrt{ 1-q}}, \frac{2}{\sqrt{ 1-q}}], \nu_q )$, thus $M_{\xi}$ is diffuse and $\{H_n^q(s_q(\xi))\Omega:n\geq 0\}$, is a total orthogonal set of vectors in $\overline{M_{\xi}\Omega}^{\norm{\cdot}_q}$. Write $\mathcal{E}_{\xi}=\{\xi^{\otimes n}:n\geq 0\}$.

\begin{Lemma}\label{Intersection}
The following hold. 
\begin{enumerate}
\item Let $\xi\in \CH_\R$ be a unit vector such that $U_t\xi=\xi$ for all $t\in \R$. Then, $\mathcal{E}_{\xi}\subseteq M_q\Omega \cap M_q^{\prime}\Omega$.
\item Let $\xi\in \CH_\R$ be a unit vector.  Then, $\overline{M_\xi\Omega}^{\norm{\cdot}_q}= \overline{\text{ span }\mathcal{E}_\xi}^{\norm{\cdot}_q}$. 
\end{enumerate}
\end{Lemma}

\begin{proof}
$(1)$ This follows directly from Lemma \ref{VectorinMq} as $\xi\in \mathfrak{D}(A^{-\half})$.

\noindent$(2)$ From the Wick product formula in Prop. $2.9$ of \cite{BKS}, it follows that $\xi^{\otimes n}=H_{n}^{q}(s_q(\xi))\Omega$ for all $n\geq 0$ $($by convention $\xi^{\otimes 0}=\Omega)$. Thus, $\xi^{\otimes n} \in M_\xi\Omega$ for all $n\geq 0$. It is now clear that $\overline{span\text{ } \mathcal{E}_\xi}^{\norm{\cdot}_q}\subseteq \overline{M_\xi\Omega}^{\norm{\cdot}_{q}}$. Now use Stone-Weierstrass and Kaplansky density theorems or the fact that $\overline{M_\xi\Omega}^{\norm{\cdot}_{q}}\cong L^{2}([-\frac{2}{\sqrt{ 1-q}}, \frac{2}{\sqrt{ 1-q}}],\nu_q)$ to establish the reverse inclusion. 
\end{proof}

\begin{Theorem}\label{eigen-vector}
Let $\xi\in \CH_\R$ be a unit vector. There exists unique $\varphi$-preserving faithful normal conditional expectation $\mathbb{E}_{\xi} : M_q \rightarrow M_{\xi}$ if and only if $s_q(\xi) \in M_q^{\varphi}$, equivalently $U_t\xi = \xi$ for all $t \in \R$.
\end{Theorem}
 
\begin{proof}
Suppose there exists a conditional expectation $\mathbb{E}_{\xi} : M_q \rightarrow M_{\xi}$ 
such that $\varphi(\mathbb{E}_{\xi}(x)) = \varphi(x)$, for all $x \in M_q$. Clearly, $\mathbb{E}_{\xi}$ is faithful and normal. 
By Takesaki's theorem \cite{Ta}, we have $\sigma_t^{\varphi}(M_{\xi}) = M_{\xi}$ for all $t\in \R$. Moreover, from  \cite{Ta} we have $ \mathbb{E}_{\xi}\circ \sigma_t^{\varphi} = \sigma_t^{\varphi} \circ \mathbb{E}_{\xi}$ for all $t\in \R$. Thus, 
\begin{align*}
\mathbb{E}_{\xi}( \sigma_t^{\varphi}(s_q(\xi)) )= \sigma_t^{\varphi} (\mathbb{E}_{\xi} (s_q(\xi))) \text{ for all }t.  
\end{align*}
Let $P_{\xi}: L^{2}(M_q,\varphi)\rightarrow \overline{M_{\xi}\Omega}^{\norm{\cdot}_q}$ denote the orthogonal projection $(L^{2}(M_q,\varphi)=\mathcal{F}_q(\CH))$. Since $\varphi(s_q(\xi))=0$, so $\varphi(\sigma_t(s_q(\xi)))=0$ for all $t\in \R$ as well. Thus, using Lemma \ref{Intersection} and expanding in terms of orthonormal basis, we have $\sigma_t^{\varphi}(s_q(\xi)) \Omega= \sum_{n = 1}^{\infty} a_n(t) \xi^{\otimes n }$, $a_n(t)\in \C$, for all $t\in \R$.  Hence, from Eq. \eqref{modularaut}, we have 
\begin{align*}
U_{-t}\xi&= s_q(U_{-t}\xi)\Omega = \sigma_{t}^{\varphi}(s_q(\xi))\Omega\\
&=\sigma_t^{\varphi}(\mathbb{E}_{\xi}(s_q(\xi)))\Omega = \mathbb{E}_{\xi} ( \sigma_t^{\varphi}(s_q(\xi)))\Omega\\
&= P_\xi \sigma_t^{\varphi}(s_q(\xi)) P_\xi \Omega\\
&= P_\xi \sum_{n = 1}^{\infty} a_n(t) \xi^{\otimes n }\\
& = \sum_{n = 1}^{\infty} a_n(t) \xi^{\otimes n}.
\end{align*}
Consequently, $a_{n}(t)=0$ for all $n\geq 2$ from Eq. \eqref{modulartheory}, and 
\begin{align*}
U_{-t} \xi = a_1(t)\xi = \lambda_{t}\xi, \text{ for all } t \in\R. 
\end{align*}
As $\Omega$ is seperating for $M_q$, it follows that $\sigma_t^{\varphi}(s_q(\xi)) = \lambda_t s_q(\xi)$. Thus, $\lambda_t\lambda_s=\lambda_{t+s}$ for all $t,s\in \R$, $\lambda_{0}=1$, $\lambda_{t}\in\{\pm 1\}$ $($as $s_q(\xi)$ is self-adjoint$)$ and $t\mapsto \lambda_t$ is continuous. Since the image of a connected set under a continuous map is connected, so either $\lambda_t=1$ for all $t$ or $\lambda_t=-1$ for all $t$. But $\lambda_0 =1$, so $\lambda_t=1$ for all $t$. Hence, $s_q(\xi)\in M_q^{\varphi}$. 

Conversely, suppose $s_q(\xi) \in M_{q}^{\varphi}$. Then $ M_\xi \subseteq M_q^{\varphi}$ and the modular group fixes $M_q^{\varphi}$ pointwise. Now use Takesaki's theorem \cite{Ta} to finish the proof. 
\end{proof}
 
We end this section with the following observation.  

\begin{Lemma}\label{fixed}
For $\eta \in M_q^{\prime}\Omega $ and $\zeta\in M_q\Omega $ one has $s_q(\zeta) \eta = d_q(\eta)\zeta$. In particular, for  $\eta\in \mathcal{Z}(M_q)\Omega$ the same holds. 
\end{Lemma}
\begin{proof}
First note that the operators in the statement are defined by Eq. \eqref{sq}. Now $s_q(\zeta)\eta = s_q(\zeta)d_q(\eta)\Omega = d_q(\eta)s_q(\zeta) \Omega=d_q(\eta)\zeta$. 
\end{proof}


\section{Strong Mixing of Generator Masas}\label{strongmixing}

In this section, we intend to show that for any unit vector $\xi_{0}\in \CH_\R$ with $U_t\xi_{0}=\xi_{0}$ for all $t\in\R$, the abelian algebra $M_{\xi_{0}}$ of $M_q$ is a masa and possess vigorous mixing properties. Needless to say, such masa is then singular from \cite{FMI77,Muk13,Jan-Fang-Kunal}. In order to do so,  we need some general facts on masas. Most of these facts appear in the literature in the framework of finite von Neumann algebras. But, the masas of our interest in $M_q$ lie in the centralizer $M_q^{\varphi}$ by Thm. \ref{eigen-vector}; so we can freely invoke most of these techniques $($used for finite von Neumann algebras$)$ in our set up as well. We recall without proofs some facts that will be required in the sequel, as a detailed exposition would be a digression. The proofs of these facts are analogous to the ones for the tracial case.

Let $M$ be a von Neumann algebra equipped with a faithful normal state $\varphi$. Let $M$ act on the GNS Hilbert space $L^{2}(M,\varphi)$ via left multiplication and let $\norm{\cdot}_{2,\varphi}$ denote the norm of $L^{2}(M,\varphi)$. Let $J_{\varphi}, \Omega_{\varphi}$ respectively denote the associated modular conjugation operator and the vacuum vector, and let $(\sigma_{t}^{\varphi})_{t\in \R}$ denote the modular automorphisms associated to $\varphi$. Let $A\subseteq M$ be a diffuse abelian von Neumann subalgebra contained in $M^{\varphi}=\{x\in M: \sigma_t^{\varphi}(x)=x \text{ }\forall \text{ }t\in\R\}$. Then there exists a unique faithful, normal and $\varphi$-preserving conditional expectation $\mathbb{E}_{A}$ from $M$ on to $A$ \cite{Ta}. Let  $L^{2}(A,\varphi)= \overline{A\Omega_{\varphi}}^{\norm{\cdot}_{2,\varphi}}$. Denote $\mathcal{A}=(A\cup J_{\varphi}AJ_{\varphi})^{\prime\prime}$. Then $\mathcal{A}$ is abelian, so its commutant is a type $\rm{I}$ algebra. 
Note that $A^{\prime}\cap M$ is globally invariant under $(\sigma_t^{\varphi})$, thus there exists a unique faithful, normal and $\varphi$-preserving conditional expectation from $M$ on to $A^{\prime}\cap M$ $($see \cite{Ta}$)$, and the associated Jones' projection $e_{A^{\prime}\cap M}\in \mathcal{A}$ \cite[Lemma 7.1.1]{SS} and is a central projection of $\mathcal{A}^{\prime}$.\footnote{The right action $J_\varphi u J_\varphi$ for $u\in \mathcal{U}(A)$ is right multiplication by $u^*$, since $u$ is analytic. Thus, the proof of \cite[Lemma 7.1.1]{SS} works out in our set up.}
$($This fact will not be directly used in this paper, nevertheless, it is worth mentioning as it is this fact for which the theory of bimodules of masas work and is indispensable$)$. This algebra $\mathcal{A}$  has been studied extensively by many experts in the context of masas to understand the size of normalizers, orbit equivalence, mixing properties and to provide invariants of masas. In short, $\mathcal{A}$ captures the structure of $L^{2}(M,\varphi)$ as a $A$-$A$ bimodule $($see Ch. 6, 7 \cite{SS}$)$. 

With the set up as above we define the following:

\begin{Definition}$($c.f. \cite{Jan-Fang-Kunal}$)$\label{sm}
The diffuse abelian algebra $A\subseteq M$ is said to be $\varphi$-strongly mixing in $M$ if $\norm{\mathbb{E}_{A}(xa_ny)}_{2,\varphi}\rightarrow 0$ for all $x,y\in M$ with $\mathbb{E}_{A}(x)=0=\mathbb{E}_{A}(y)$, whenever $\{a_n\}$ is a bounded sequence in $A$ that goes to $0$ in the $w.o.t$.
\end{Definition}
In fact, by a polarization identity it is enough to check the convergence of $\mathbb{E}_{A}(xa_nx^{*})$ in Defn. \ref{sm} for all $x\in M$ such that $\mathbb{E}_{A}(x)=0$. 

Let $M_{a}$ denote the $*$-subalgebra of all entire $($analytic$)$ elements of $M$ with respect to $(\sigma_t^{\varphi})$. For $x\in M$ and $y\in M_a$, define
\begin{align}\label{T-operator}
T_{x,y}: L^{2}(A,\varphi)\rightarrow L^{2}(A,\varphi) \text{ by } T_{x,y}(a\Omega_{\varphi})=\mathbb{E}_{A}(xay)\Omega_{\varphi}, \text{ }a\in A. 
\end{align}
Note that $T_{x,y}$ is bounded. Indeed, as $y\in M_a$ so $y^{*}\in \mathfrak{D}(\sigma_{z}^{\varphi})$ for all $z\in\mathbb{C}$. Hence, $J_{\varphi}\sigma_{-\frac{i}{2}}^{\varphi}(y^*)J_{\varphi}a\Omega_{\varphi} = ay\Omega_{\varphi}$ for all $a\in A$, where $(\sigma_{z}^{\varphi})_{z\in \C}$ denotes the analytic continuation of $(\sigma_{t}^{\varphi})$ $($see \cite{Fal00}$)$. Thus, 
\begin{align*}
\norm{\mathbb{E}_{A}(xay)\Omega_{\varphi}}_{2,\varphi}&\leq \norm{xay\Omega_{\varphi}}_{2,\varphi}\\
&\leq \norm{x}\norm{ay\Omega_{\varphi}}_{2,\varphi}\\
&\leq\norm{x}\norm{J_{\varphi}\sigma_{-\frac{i}{2}}^{\varphi}(y^*)J_{\varphi}}\norm{a\Omega_{\varphi}}_{2,\varphi}\\
&= \norm{x}\norm{\sigma_{-\frac{i}{2}}^{\varphi}(y^*)}\norm{a\Omega_{\varphi}}_{2,\varphi}, \text{ for all }a\in A.
\end{align*} 

One can identify $A\cong L^{\infty}(X,\lambda)$, where $X$ is a standard Borel space and 
$\lambda$ is a non-atomic probability measure on $X$. The \textit{left-right} measure of $A$ is the measure $($strictly speaking the measure class$)$ on $X\times X$ obtained from the direct integral decomposition of $L^{2}(M,\varphi)\ominus L^{2}(A,\varphi)$ over $X\times X$ so that $\mathcal{A}$ is the algebra of diagonalizable operators with respect to the decomposition \cite[Defn. 3]{Muk13}, \cite{Jan-Fang-Kunal}. If $A$ is identified with $L^{\infty}([a,b],\lambda)$ where $\lambda$ is the normalized Lebesgue measure $($or Lebesgue equivalent$)$, then from the results of \S2 of \cite{Muk13} $($specifically Thm. 2.1$)$, it follows that the \textit{left-right} measure of $A$ is \textit{Lebesgue absolutely continuous} when $T_{x,y^{*}}$ is Hilbert-Schmidt for $x,y$ varying over a set $S$ such that $\mathbb{E}_{A}(x)=0=\mathbb{E}_{A}(y)$ for all $x,y\in S$ and the span of $S\Omega$ is dense in ${L^{2}(A,\varphi)}^{\perp}$. $($Note that the arguments of \S2 of \cite{Muk13} use the unit interval. It was so chosen  to make a standard frame of reference. However, the arguments of \S2 relating to absolute continuity of measures do not depend on the choice of the interval. Neither does the same arguments to prove Thm. 2.1 in \cite{Muk13} required that $A$ is a masa; it only involved measure theory relevant to the context.$)$ From Thm. 4.4 and Rem. 4.5 of \cite{CFM2} $($similarly the proof of Thm. 4.4 of \cite{CFM2} uses measure theory and not that the diffuse abelian algebra there is a masa$)$, it follows that $A$ is $\varphi$-strongly mixing in $M$ if the \textit{left-right} measure of $A$ is Lebesgue absolutely continuous. Thus, one has:

\begin{Theorem}\label{AllHilbertSchmidt}
Let $A\subseteq M$ be a diffuse abelian algebra such that $A\subseteq M^\varphi$ and the left-right measure of $A$ is Lebesgue absolutely continuous. Then, $A$ is $\varphi$-strongly mixing in $M$. In particular, $A$ is a singular masa in $M$.
\end{Theorem}

\begin{proof}
We only need to show that $A$ is a singular masa in $M$. Let $x\in A^{\prime}\cap M$. Let $y=x-\mathbb{E}_{A}(x)$. Since $A$ is diffuse choose a sequence of unitaries $u_n\in A$ such that $u_n\rightarrow 0$ in $w.o.t$. Thus, by the previous discussion and by the hypothesis, it follows that $\norm{yy^*}_{2,\varphi}=\lim_{n}\norm{\mathbb{E}_{A}(yu_n y^*)}_{2,\varphi}=0$. Thus, $y=0$ proving $A$ is a masa. 

That $A$ is singular follows from results of \cite{FMI77,FMII77}and \cite{Muk09}.
\end{proof}

We are now ready to prove that if $\xi_0\in \CH_\R$ is a unit vector such that $U_t\xi_0=\xi_0$ for all $t\in\R$, then $M_{\xi_0}$ is $\varphi$-strongly mixing in $M_q$. Let $\mathbb{E}_{\xi_0}$ denote the unique $\varphi$-preserving, faithful, normal conditional expectation from $M_q$ onto $M_{\xi_0}$ $($see Thm. \ref{eigen-vector}$)$. Extend $\xi_0$ to an orthonormal basis 
\begin{align*}
\mathcal{O}=\{\xi_k:\xi_{k} \text{ analytic, } 0\leq k\leq dim(\CH_\R)-1\}
\end{align*}
of $\CH_\R$ with respect to $\langle\cdot,\cdot\rangle_{\CH_\C}$ consisting of analytic vectors 
as described in Prop. \ref{analytic}. Fix $\xi_{i_j}\in \mathcal{O}$ for $1\leq j\leq n$. Note that as the analytic elements form a $(w^*$-dense$)$ $*$-subalgebra, so $s_q(\xi_{i_1}\otimes\cdots\otimes \xi_{i_n})$ is analytic with respect to $(\sigma_t^{\varphi})$ from $($the proof of$)$ Lemma \ref{VectorinMq}. Again by Rem. \ref{analytic_extension_basis}, it follows that $s_q(A^{-\half}\xi_k)$ is also analytic with respect to $(\sigma_t^{\varphi})$ for all $\xi_k\in \mathcal{O}$. Thus, by $($the proof of$)$ Lemma \ref{VectorinMq}, it follows that $s_q(A^{-\half}\xi_{i_1}\otimes\cdots\otimes A^{-\half}\xi_{i_n})$ is also analytic with respect to $(\sigma_t^{\varphi})$. Moreover, from Lemma \ref{Intersection} and Lemma \ref{Orthogonality} it follows that  $\mathbb{E}_{\xi_0}(s_q(\xi_{i_1}\otimes\cdots\otimes \xi_{i_n}))=0$ forces that at least one letter $\xi_{i_j}$ must be different from $\xi_0$. Furthermore, from Lemma \ref{Orthogonality} it follows that $A^{-\half}\xi_{i_1}\otimes\cdots\otimes A^{-\half}\xi_{i_n}\in \mathcal{F}_q(\CH)\ominus L^{2}(M_{\xi_0},\varphi)$ if and only if  $\xi_{i_1}\otimes\cdots\otimes \xi_{i_n}\in \mathcal{F}_q(\CH)\ominus L^{2}(M_{\xi_0},\varphi)$. 

In the light of the above discussion, the following theorem is crucial for our purpose.		
											
\begin{Theorem}\label{stronglymixingmasa}
Let $t\mapsto U_t$ be a strongly continuous orthogonal representation of $\R$ on a real Hilbert space $\CH_\R$ with $dim(\CH_\R)\geq 2$. Suppose there exists a unit vector $\xi_0\in\CH_\R$ such that $U_t\xi_0=\xi_0$ for all $t\in\R$. Let $x = s_q(\xi_{i_1}\otimes\cdots\otimes\xi_{i_m})$ and  $y=s_q( A^{-\half}\xi_{j_1}\otimes\cdots\otimes A^{-\half} \xi_{j_k})$ be such that $\mathbb{E}_{\xi_0}(x)=0=\mathbb{E}_{\xi_0}(y)$, where $\xi_{i_u}, \xi_{j_v}\in \mathcal{O}$ for $1\leq u\leq m$ and $1\leq v\leq k$. Then, $T_{x,y}$ is a Hilbert-Schmidt operator.
\end{Theorem}

\begin{proof}
First of all, as $U_t\xi_0=\xi_0$ for all $t\in\R$, so $M_{\xi_0}\subseteq M_q$ is a diffuse abelian algebra in $M_q$ lying in $M_q^{\varphi}$. By the previous discussion, it follows that $x,y$ are analytic with respect to $(\sigma_t^{\varphi})$. Thus, $T_{x,y}\in \mathbf{B}(L^{2}(M_{\xi_0},\varphi))$.

Also note that  $\xi_{j_1}\otimes\cdots\otimes\xi_{j_k}\in M_q\Omega\cap M_q^{\prime}\Omega$ from Lemma \ref{VectorinMq}. From Lemma \ref{Intersection}, it follows that $H_n^q(s_q(\xi_0))\Omega=\xi_0^{\otimes n}$ for all $n\geq 0$. Note that $d_q( A^{-\half}\xi_{j_1}\otimes\cdots\otimes A^{-\half} \xi_{j_k})\in M_q^{\prime}$ by Thm. \ref{commutant}. Let $e_{\xi_0}:L^{2}(M_q,\varphi)\rightarrow L^{2}(M_{\xi_0},\varphi)$ denote the Jones' projection associated to $M_{\xi_0}$. Then from Eq. \eqref{T-operator}, we have 
\begin{align}\label{T-evaluateonwords}
T_{x, y}\left(H^q_n( s_q(\xi_0))\Omega\right)&=e_{\xi_0}\left(x H^q_n(s_q(\xi_0))s_q(A^{-\half}\xi_{j_1}\otimes\cdots\otimes A^{-\half}\xi_{j_k})\Omega\right) \\
\nonumber&= e_{\xi_0}\left(x H^q_n(s_q(\xi_0))(A^{-\half}\xi_{j_1}\otimes\cdots\otimes A^{-\half}\xi_{j_k})\right) \\
\nonumber&= e_{\xi_0}\left(x H^q_n(s_q(\xi_0)) d_q (A^{-\half}\xi_{j_1} \otimes \cdots \otimes A^{-\half}\xi_{j_k})\Omega\right) \\
\nonumber&\indent\indent\indent\indent\indent\indent\indent\indent\indent\indent\text{ (from Eq. \eqref{sq} and Lemma \ref{fixed})}\\
\nonumber&= e_{\xi_0}\left( x d_q(A^{-\half}\xi_{j_1}\otimes\cdots\otimes A^{-\half} \xi_{j_k})H^q_n(s_q(\xi_0))\Omega\right) \\
\nonumber&= e_{\xi_0}\left(x d_q(A^{-\half}\xi_{j_1}\otimes\cdots\otimes A^{-\half}\xi_{j_k})\xi_{0}^{ \otimes n}\right)\\
\nonumber&= e_{\xi_0}\left(s_q(\xi_{i_1}\otimes\cdots\otimes \xi_{i_m}) d_q(A^{-\half}\xi_{j_1}\otimes\cdots\otimes A^{-\half}\xi_{j_k})\xi_0^{ \otimes n}\right), \text{ }n\geq 0.
\end{align} 
Now from Lemma 3.1 of \cite{Hiai}, we have 
\begin{align*}
&s_q(\xi_{i_1}\otimes\cdots \otimes \xi_{i_m})
= \sum \sum q^{\aleph(K, I )}c_q(\xi_{i_{\kappa(1)}}) \cdots c_q(\xi_{i_{\kappa(n_1)}}) c_q(\xi_{i_{\pi(1)}})^*\cdots c_q(\xi_{i_{ \pi(n_2)}})^{*} \text{ and}\\
&d_q(A^{-\half}\xi_{j_1}\otimes \cdots \otimes A^{-\half}\xi_{j_k})\\
&~~~~~= \sum \sum q^{\aleph(K^\prime, I^\prime )}r_q(A^{-\half}\xi_{j_{\tilde\kappa(1)}}) \cdots r_q(A^{-\half}\xi_{j_{\tilde\kappa(m_1)}}) r_q(A^{-\half}\xi_{j_{\tilde\pi(1)}})^{*}\cdots r_q(A^{-\half}\xi_{j_{\tilde\pi(m_2)}})^{*},
\end{align*}   
where the first sum varies over the pairs $(n_1, n_2) $ and $(K, I)$ restricted to the following conditions:
\begin{align}\label{Constraints}{\begin{array}{cc}n_1,n_2 \geq 0,\\
n_1+n_2 = m; \end{array}} \text { and, } {\begin{array}{ccc} K = \{ \kappa(1), \cdots , \kappa(n_1): \kappa(1) \leq \cdots \leq\kappa(n_1)\},\\ I = \{ \pi(1), \cdots , \pi(n_2): \pi(1) \leq \cdots \leq \pi(n_2)\},\\ K \cup I = \{ 1, \cdots , m\}, K \cap I = \emptyset,
\end{array}}
\end{align} and $\aleph(K,I) = \# \{ (r, s) : 1 \leq r \leq n_1, 1 \leq s\leq n_2, \kappa (r) > \pi(s) \}$. Similarly, the expansion of $d_q(A^{-\half}\xi_{j_1} \otimes  \cdots \otimes A^{-\half}\xi_{j_k})$ above is in terms of $m_1,m_2\geq 0$, $m_1+m_2=k$, $K^{\prime}, I^{\prime}, \aleph(K^{\prime},I^{\prime})$, $\tilde\kappa,\tilde\pi$ and $r_q(A^{-\half}\xi_{j_{\tilde\kappa(\cdot)}}) $ and $r_q(A^{-\half}\xi_{j_{\tilde\pi(\cdot)}})^* $ defined analogous to Eq. \eqref{Constraints}.

Note that $\norm{\xi_0^{\otimes n}}_{q}^{2}=[n]_{q}!$ for all $n\geq 0$ $($see Eq. \eqref{Normelt}$)$. Again from Lemma \ref{Intersection}, it follows that $\{\frac{1}{\sqrt{[n]_{q}!}}\xi_0^{\otimes n}: n\geq 0\}$ is an orthonormal basis of $L^{2}(M_{\xi_0},\varphi)$. Thus, to show $T_{x,y}$ is a Hilbert-Schmidt operator we need to show that $\sum_{n=0}^{\infty} \frac{1}{[n]_{q}!}\norm{T_{x,y}(\xi_0^{\otimes n})}_{q}^{2}<\infty$. But since $s_q(\xi_{i_1}\otimes\cdots \otimes \xi_{i_m})$ and $d_q(A^{-\half}\xi_{j_1}\otimes \cdots \otimes A^{-\half}\xi_{j_k})$ split as finite sums, so from Eq. \eqref{T-evaluateonwords} it is enough to show that for each fixed $n_1,n_2, m_1,m_2,\kappa,\pi,\tilde{\kappa},\tilde{\pi}$ $($in Eq. \eqref{Constraints}$)$, if
\begin{align*}
\zeta_n&=e_{\xi_0}(\Big(c_q(\xi_{i_{\kappa(1)}}) \cdots c_q(\xi_{i_{\kappa(n_1)}}) c_q(\xi_{i_{\pi(1)}})^*\cdots c_q(\xi_{i_{ \pi(n_2)}})^{*}\\
&\indent\cdot  r_q(A^{-\half}\xi_{j_{\tilde\kappa(1)}}) \cdots r_q(A^{-\half}\xi_{j_{\tilde\kappa(m_1)}}) r_q(A^{-\half}\xi_{j_{\tilde\pi(1)}})^{*}\cdots r_q(A^{-\half}\xi_{j_{\tilde\pi(m_2)}})^{*}\Big)\xi_0^{\otimes n}), n\geq 0,
\end{align*}
then $\sum_{n=0}^{\infty}\frac{1}{[n]_{q}!}\norm{\zeta_n}_q^{2}<\infty$. Renaming indices, we may write 
\begin{align*} 
\zeta_{n}=e_{\xi_0}(\Big(c_q(\xi_{i_{1}})& \cdots c_q(\xi_{i_l}) c_q(\xi_{i_{l+1}})^*\cdots c_q(\xi_{i_{m}})^{*}\\
&\cdot r_q(A^{-\half}\xi_{j_{1}}) \cdots r_q(A^{-\half}\xi_{j_{p}}) r_q(A^{-\half}\xi_{j_{p+1}})^{*}\cdots r_q(A^{-\half}\xi_{j_{k}})^{*}\Big)\xi_0^{\otimes n}), \text{ }n\geq 0.
\end{align*}

For $\xi_{j^{\prime}}\in\mathcal{O}$, since $\langle\xi_{j^{\prime}}, \xi_0\rangle_q=0$ for $j^{\prime}\neq 0$ $($by Lemma \ref{Orthogonality}$)$, $($and hence $\lan A^{-\half}\xi_0,  A^{-\half}\xi_{j^{\prime}}\ran_q = 0$ for $j^{\prime}\neq  0$ by Eq. \eqref{A-Inner-1}$)$, so $r_q(A^{-\half}\xi_{j^{\prime}})^*\xi_0^{ \otimes n} = r_q(A^{-\half}\xi_{j^{\prime}})^*(A^{-\half}\xi_0)^{ \otimes n}=0$ for all $n\geq 0$ and $j^{\prime}\neq 0$.
Since at least one letter in $A^{-\half}\xi_{j_1}\otimes\cdots\otimes A^{-\half}\xi_{j_k}$ is different from $\xi_0$ and $A^{-\half} \xi_0 = \xi_0$, so $\zeta_n$ can be non zero only when $\xi_{j_{p+1}}=\cdots=\xi_{j_k}=\xi_0$. Write $\delta=\prod_{w=p+1}^{k}\delta_{\xi_{j_{w}},\xi_0}$. Hence, from Eq. \eqref{Rightmult} and Eq. \eqref{Rightmultadj} we have 
\begin{align}\label{NormofZeta}
\zeta_n&=\delta\prod_{t= n-(k-p)}^{n}(1+ q+ \cdots + q^{t-1})\\
\nonumber&\indent\cdot e_{\xi_0}\Big(c_q(\xi_{i_1})\cdots c_q(\xi_{i_{l}})
c_q(\xi_{i_{l+1}})^*\cdots c_q(\xi_{i_{m}})^*\big(\xi_{0}^{\otimes (n -(k-p))}\otimes 
                 A^{-\half}\xi_{j_1}\otimes\cdots\otimes A^{-\half}\xi_{j_{p}}\big)\Big)\\
\nonumber&=\delta\frac{[n]_q!}{[n-(k-p)]_q!}e_{\xi_0}\Big(c_q(\xi_{i_1})\cdots c_q(\xi_{i_{l}})
c_q(\xi_{i_{l+1}})^*\cdots c_q(\xi_{i_{m}})^*\big(\xi_{0}^{\otimes (n -(k-p))}\otimes A^{-\half}\xi_{j_1}\otimes\cdots\otimes  A^{-\half}\xi_{j_{p}}\big)\Big).
\end{align}
By hypothesis at least one letter in $A^{-\half}\xi_{j_1}\otimes\cdots\otimes A^{-\half}\xi_{j_{p}} $ is different from $\xi_0\text{ }(=A^{-\half}\xi_0)$. Therefore, the constraints for $\zeta_n$ to be non zero are  $i_r=0$ for all $1\leq r\leq l$, $\#\{i_r:l+1\leq r\leq m, i_r\neq 0\}\geq 1$ $($counted with multiplicities$)$ and the expression 
\begin{align*}
c_q(\xi_{i_1})\cdots  c_q(\xi_{i_{l}})c_q(\xi_{i_{l+1}})^*\cdots c_q(\xi_{i_{m}})^*(\xi_0^{\otimes (n -(k-p))} \otimes 
A^{-\half}\xi_{j_1}\otimes \cdots \otimes A^{-\half}\xi_{j_{p}})
\end{align*}
has to lie in $\text{span }\mathcal{E}_{\xi_0}$ $($see Lemma \ref{Intersection} and the discussion preceding it$)$.  

By repeated application of Lemma \ref{SplitAdjoint}, one obtains 
\begin{align}\label{Expanding}
&c_q(\xi_{i_{l+1}})^*\cdots c_q(\xi_{i_{m}})^*\left(\overbrace{\xi_{0}^{\otimes (n -(k-p))}}\otimes \overbrace{(A^{-\half}\xi_{j_1}\otimes\cdots\otimes  A^{-\half}\xi_{j_{p}})}\right)\\
\nonumber=& c_q(\xi_{i_{l+1}})^*\cdots c_q(\xi_{i_{m-1}})^*
\Bigg(\overbrace {(c_q(\xi_{i_{m}})^*\xi_{0}^{\otimes (n -(k-p))})}\otimes \overbrace{(A^{-\half}\xi_{j_1}\otimes\cdots\otimes  A^{-\half}\xi_{j_{p}})}\\
\nonumber&\indent\indent\indent\indent\indent\indent\indent\indent\indent+  q^{(n -(k-p))}\overbrace{\xi_{0}^{\otimes (n -(k-p))}}\otimes 
\overbrace{c_q(\xi_{i_{m}})^* (A^{-\half}\xi_{j_1}\otimes\cdots\otimes  A^{-\half}\xi_{j_{p}})} \Bigg)\\
\nonumber&\vdots\\
\nonumber=&\sum_{r_1=0}^1\cdots \sum_{r_{m-l}=0}^1c_{r_1,\cdots, r_{m-l}}\cdot\\
\nonumber&\indent\indent\indent\Bigg(\prod_{w=1}^{m-l}{\big(c_q(\xi_{i_{l+w}})^*\big)}^{(1-r_{w})}\Bigg)\xi_0^{\otimes (n -(k-p))}
\otimes \Bigg(\prod_{w=1}^{m-l}{\big(c_q(\xi_{i_{l+w}})^*\big)}^{r_{w}}\Bigg)(A^{-\half}\xi_{j_1}\otimes\cdots\otimes  A^{-\half}\xi_{j_{p}}),
\end{align}
where $c_{r_1,\cdots, r_{m-l}}\in\R$ for $(r_1,\cdots, r_{m-l})\in\{0,1\}^{m-l}$ are calculated as follows. 

Given a $(m-l)$-bit string $(r_1,\cdots, r_{m-l})$, let $s_w=\# \text{ of zeros in }\{r_w,r_{w+1},\cdots, r_{m-l}\}$ for $1\leq w\leq m-l$. Then, clearly $s_{m-l}=1-r_{m-l}$ and by induction it follows that $s_{m-l-1}=(1-r_{m-l})+(1-r_{m-l-1})$, $\cdots$, $s_1=(1-r_{m-l})+(1-r_{m-l-1})+\cdots+(1-r_1)$. Thus, repeated application of Lemma \ref{SplitAdjoint} in Eq. \eqref{Expanding} entail that 
\begin{align*}
c_{r_1,\cdots, r_{m-l}}&=q^{(n -(k-p))\Big(\sum_{w=1}^{m-l}r_w\Big)-\sum_{w=1}^{m-l}r_ws_w}\\
\nonumber&=q^{(n -(k-p))\Big(\sum_{w=1}^{m-l}r_w\Big)-\sum_{w=1}^{m-l}r_w\Big((m-l)-w+1-\sum_{w^{\prime}=w}^{m-l}r_{w^{\prime}}\Big)}\\
\nonumber&=q^{\big((n -(k-p))-(m-l)-1\big)\Big(\sum_{w=1}^{m-l}r_w\Big)+\sum_{w=1}^{m-l}wr_w+\sum_{w=1}^{m-l}\Big(\sum_{w^{\prime}=w}^{m-l}r_{w^{\prime}}\Big)r_w}.
\end{align*}
The above formula for $c_{r_1,\cdots, r_{m-l}}$ can be obtained by drawing a binary tree of height $(m-l)$ with weights attached along edges in such a way that it encodes the tensoring on the left or on the right following Lemma \ref{SplitAdjoint}. It is to be noted that the largest power of $q$ that appears in Eq. \eqref{Expanding} is $(n -(k-p))(m-l)$ which appears when $r_w=1$ for all $w$ and the smallest power of $q$ is $0$ and it occurs when $r_w=0$ for all $w$. 

Further, notice that since $\#\{i_r:l+1\leq r\leq m, i_r\neq 0\}\geq 1$, i.e., there is at least one $r_0 $ with $ l+1 \leq r_0 \leq m$ such that $ \xi_{i_{r_0}} \perp \xi_0$ $($in $\langle\cdot,\cdot\rangle_U)$, so  
\begin{align*}
\Big(c_q(\xi_{i_{l+1}})^*\cdots 
c_q(\xi_{i_{m-1}})^* c_q(\xi_{i_{m}})^*\Big)\xi_{0}^{\otimes (n -(k-p))}\otimes \Big(A^{-\half}\xi_{j_1}\otimes\cdots\otimes  A^{-\half}\xi_{j_{p}}\Big) = 0. 
\end{align*}
Therefore, the expression in Eq. \eqref{Expanding} has at most $2^{m-l-1}$ many non zero terms each with scalar coefficients of the form $ q^d$, where $d \geq{((n -(k-p))-(m-l-1))}$. Consequently, by Eq. \eqref{normestimate}, Eq. \eqref{Leftmult}, Eq. \eqref{flip} and Eq. \eqref{NormofZeta}, we conclude that there is a positive constant $K(l,m,p,q)$ independent of $n$ and $N_0\in \N$ such that 
\begin{align*}
\norm{\zeta_n}_q^2 \leq K(l, m, p, q)q^{2n} \bigg(\frac{[n]_{|q|}!}{[n-(k-p)]_{|q|}!}
\sqrt{[n-N_0]_{q}!}\bigg)^2, \text{ for all }n > N_0.
\end{align*}

Define a sequence $\{a_n\}$ of real numbers as follows:
\begin{equation*}
a_n =\begin{cases}
      1, \quad &\text{if }0\leq n\leq N_0,\\
      \frac{1}{[n]_{q}!}{|q|}^{2n}\bigg(\frac{[n]_{|q|}!}{[n-(k-p)]_{|q|}!}
\sqrt{[n-N_0]_{q}!}\bigg)^2,\quad &\text{otherwise}.      
     \end{cases}
\end{equation*}
Note that $\text{lim}_{ n \rightarrow\infty} \frac{a_{n+1}}{a_n}=\abs{q}^2<1$. Consequently, by ratio test $ \sum_{n\geq 1}a_n<\infty$. Since the sequence $\{a_n\}$ eventually dominates the tail of the sequence $\{\frac{1}{[n]_{q!}}\norm{\zeta_n}_q^2\}$ modulo a scalar multiple, the proof is complete.
\end{proof}

Thus, we have the following results.

\begin{Theorem}\label{masa}
Let $t\mapsto U_t$ be a strongly continuous orthogonal representation of $\R$ on a real Hilbert space $\CH_\R$ with $dim(\CH_{\R}) \geq 2$. Let $\xi_0\in \CH_\R $ be a unit vector such that $U_t \xi_0=\xi_0$ for all $t\in \R$. Then,  $M_{\xi_0}$ is a $\varphi$-strongly mixing masa in $M_q$ whose left-right measure is Lebesgue absolutely continuous. 
\end{Theorem}

\begin{proof}
In this proof, we repeatedly use Eq. \eqref{modulartheory}, the right multiplication of elements of $M_{\xi_0}$ from \cite{Fal00} and the fact that the analytic extension of $(\sigma_t^{\varphi})$ is algebraic on the analytic elements of $M_q$. Fix $m,p\in \N$. Note that if $\xi_{i_1},\cdots,\xi_{i_m}\in \mathcal{O}$ and $\xi_{j_1},\cdots,\xi_{j_p}\in\mathcal{O}$, and $x=s_q(\xi_{i_1}\otimes\cdots\otimes\xi_{i_m})$ and $y=s_q(A^{-\frac{1}{2}}\xi_{j_1}\otimes\cdots\otimes A^{-\frac{1}{2}}\xi_{j_p})$ be such that $\mathbb{E}_{\xi_0}(x)=0=\mathbb{E}_{\xi_0}(y)$, then by Thm. \ref{stronglymixingmasa} it follows that $T_{x,y}, T_{x^{*},y}$ are Hilbert-Schmidt operators. Consequently, letting $a=\frac{2}{\sqrt{1-q}}$, there exists $f\in L^{2}(\nu_q\otimes \nu_q)$ such that for all $n,k\geq 0$ one has
\begin{align*}
&\int_{-a}^{a}\int_{-a}^{a} H_n^q(t)H_k^q(s)f(t,s)d\nu_q(t)d\nu_q(s)\\
=&\big\langle s_q(\xi_{i_1}\otimes\cdots\otimes \xi_{i_m})\Omega, H_n^q(s_q(\xi_0))s_q(A^{-\frac{1}{2}}\xi_{j_1}\otimes\cdots\otimes A^{-\frac{1}{2}}\xi_{j_p}) H_k^q(s_q(\xi_0))\Omega\big\rangle_q\\
=&\big\langle s_q(\xi_{i_1}\otimes\cdots\otimes \xi_{i_m})\Omega, H_n^q(s_q(\xi_0))s_q(A^{-\frac{1}{2}}\xi_{j_1}\otimes\cdots\otimes A^{-\frac{1}{2}}\xi_{j_p}) JH_k^q(s_q(\xi_0))J\Omega\big\rangle_q\\
=&\big\langle s_q(\xi_{i_1}\otimes\cdots\otimes \xi_{i_m})\Omega, H_n^q(s_q(\xi_0))JH_k^q(s_q(\xi_0))J s_q(A^{-\frac{1}{2}}\xi_{j_1}\otimes\cdots\otimes A^{-\frac{1}{2}}\xi_{j_p}) \Omega\big\rangle_q\\
=&\big\langle H_n^q(s_q(\xi_0))s_q(\xi_{i_1}\otimes\cdots\otimes \xi_{i_m})H_k^q(s_q(\xi_0))\Omega, s_q(A^{-\frac{1}{2}}\xi_{j_1}\otimes\cdots\otimes A^{-\frac{1}{2}}\xi_{j_p})\Omega\big\rangle_q\\
=&\big\langle H_n^q(s_q(\xi_0))s_q(\xi_{i_1}\otimes\cdots\otimes \xi_{i_m})H_k^q(s_q(\xi_0))\Omega,  \Delta^{\frac{1}{2}}(\xi_{j_1}\otimes\cdots\otimes\xi_{j_p})\big\rangle_q \text{\indent(by Eq. \eqref{modulartheory})}\\
=&\big\langle \Delta^{\frac{1}{4}}\Bigg(H_n^q(s_q(\xi_0))s_q(\xi_{i_1}\otimes\cdots\otimes \xi_{i_m})H_k^q(s_q(\xi_0))\Bigg)\Omega, \Delta^{\frac{1}{4}}(\xi_{j_1}\otimes\cdots\otimes\xi_{j_p})\big \rangle_q\\
=&\big\langle\sigma_{-\frac{i}{4}}^{\varphi}\Bigg(H_n^q(s_q(\xi_0))s_q(\xi_{i_1}\otimes\cdots\otimes \xi_{i_m})H_k^q(s_q(\xi_0))\Bigg)\Omega, \Delta^{\frac{1}{4}}(\xi_{j_1}\otimes\cdots\otimes\xi_{j_p})\big\rangle_q\\
=&\big\langle H_n^q(s_q(\xi_0))\sigma_{-\frac{i}{4}}^{\varphi}\Big(s_q(\xi_{i_1}\otimes\cdots\otimes \xi_{i_m})\Big)H_k^q(s_q(\xi_0))\Omega, \Delta^{\frac{1}{4}}(\xi_{j_1}\otimes\cdots\otimes\xi_{j_p})\big\rangle_q \text{ (as }s_q(\xi_0)\in M_q^{\varphi})\\ 
=&\big\langle\sigma_{-\frac{i}{4}}^{\varphi}\Big(s_q(\xi_{i_1}\otimes\cdots\otimes \xi_{i_m})\Big)\Omega,
H_n^q(s_q(\xi_0))\Big(\sigma_{-\frac{i}{4}}^{\varphi}(s_q(\xi_{j_1}\otimes\cdots\otimes \xi_{j_p})\Big)H_k^q(s_q(\xi_0))\Omega\big\rangle_q.
%
%
%
%
\end{align*}
From the above argument, it follows that $T_{z^{*},w}$ is also a Hilbert-Schmidt operator, where $z=\sigma_{-\frac{i}{4}}^{\varphi}(s_q(\xi_{i_1}\otimes\cdots\otimes \xi_{i_m}))$ and $w=\sigma_{-\frac{i}{4}}^{\varphi}(s_q(\xi_{j_1}\otimes\cdots\otimes \xi_{j_p}))$, as it is an integral operator given by a square integrable kernel. 

Now use the discussion preceding Thm. \ref{stronglymixingmasa}, Eq. \eqref{modulartheory} and the fact that the complex span of 
\begin{align*}
\{\sigma_{-\frac{i}{4}}^{\varphi}(s_q(\xi_{i_1}\otimes\cdots\otimes \xi_{i_m})): \xi_{i_j}\in\mathcal{O}, 1\leq j\leq m, \xi_{i_j}\neq \xi_0 \text{ for at least one }\xi_{i_j}, m\in \N\} 
\end{align*}
is dense in $\mathcal{F}_q(\CH)\ominus L^{2}(M_{\xi_0},\varphi)$ to conclude that the left-right measure of $M_{\xi_0}$ is Lebesgue absolutely continuous. The rest is immediate from Thm. \ref{AllHilbertSchmidt}.
\end{proof}

The results of this section obtained so far can thus be summarized as follows. 

\begin{Corollary}\label{Corollary}
Let $t\mapsto U_t$ be a strongly continuous orthogonal representation of $\R$ on a real Hilbert space $\CH_\R$ with $dim(\CH_{\R}) \geq 2$. Let $\xi_0 \in \CH_\R$ be a unit vector. Then the following are equivalent:
\begin{enumerate}
\item $s_q(\xi_0) \in M_q^{\varphi}$$;$
\item $U_t \xi_0 = \xi_0$, for all $ t\in \R$$;$
\item there exists a faithful normal conditional expectation  $\mathbb{E}_{\xi_0}: M_q \rightarrow M_{\xi_0}$ such that $\varphi(\mathbb{E}_{\xi_0} ( x)) = \varphi(x)$ for all $x \in M_q$$;$
\item $M_{\xi_0}$ is a $\varphi$-strongly mixing masa in $M_q$.
\end{enumerate}
\end{Corollary}

\begin{proof}
The stated conditions in the statement are equivalent from Thm. \ref{eigen-vector} and Thm. \ref{masa}.
\end{proof}

Hiai proved that if the almost periodic part of the orthogonal representation is infinite dimensional, then the centralizer 
$M_q^{\varphi}$ has trivial relative communtant, i.e., $(M_q^{\varphi})^{\prime}\cap M_q =\C1$ $($Thm. 3.2 \cite{Hiai}$)$. Now we show that the same result is true under a weaker hypothesis as well.

\begin{Corollary}\label{Trivialrelcommutant}
Let $t\mapsto U_t$ be a strongly continuous orthogonal representation of $\R$ on a real Hilbert space $\CH_\R$ with $dim(\CH_{\R}) \geq 2$. Suppose there exist unit vectors $\xi_i\in\CH_\R$ such that $U_t\xi_i=\xi_i$, $i=1,2$, for all $t\in\R$, and $\langle \xi_1,\xi_2\rangle_U=0$. Then, $(M_q^{\varphi})^{\prime}\cap M_q =\C1$. 
\end{Corollary}

\begin{proof}
By Thm. \ref{eigen-vector} and Thm. \ref{masa}, it follows that $M_{\xi_i}\subseteq M_q^{\varphi}$ is a masa in $M_q$ for $i=1,2$. Let $x\in (M_q^{\varphi})^{\prime}\cap M_q$. Then $x\in M_{\xi_1}\cap M_{\xi_2}$ and hence $x\Omega\in \overline{span\text{ }\mathcal{E}_{\xi_1}}^{\norm{\cdot}_q}\cap \overline{span\text{ }\mathcal{E}_{\xi_2}}^{\norm{\cdot}_q}$ from Lemma \ref{Intersection}. But from Eq. \eqref{qFock}, it follows that $\overline{span\text{ }\mathcal{E}_{\xi_1}}^{\norm{\cdot}_q}\cap \overline{span\text{ }\mathcal{E}_{\xi_2}}^{\norm{\cdot}_q}=\C\Omega$. As $\Omega$ is a separating vector for $M_q$ the result follows.
\end{proof}

\begin{Remark}
The hypothesis of Cor. \ref{Trivialrelcommutant} actually forces $M_q$ to be a factor but we will establish the factoriality of $M_q$ under a weaker hypothesis.  
\end{Remark}

\section{Factoriality}\label{Factoriality}

In this section, we extend the previous efforts to decide the factoriality of $M_q$. We establish that $M_q$ is a factor when $dim(\CH_\R)\geq 2$ and $(U_t)$ is not ergodic or has a nontrivial weakly mixing component.

Our approach to prove factoriality is fundamentally along the lines of \'{E}ric Ricard \cite{ER}. As discussed in the introduction, our approach is to use ideas coming from Ergodic theory, namely, strong mixing, as seen in the previous section. Our idea stems from the following observation. If a finite von Neumann algebra contains a diffuse masa for which the orthocomplement of the associated Jones' projection is a coarse bimodule, then the von Neumann algebra must be a factor \cite{CFM2}. But for the masa $M_{\xi_0}$ $($in \S\ref{strongmixing}$)$, instead of showing that the orthocomplement of the Jones' projection is a coarse bimodule over $M_{\xi_0}$, we only settled with absolute continuity in Thm. \ref{stronglymixingmasa} and Thm. \ref{masa} to avoid cumbersome calculations. In this section, we use the fact that $M_{\xi_0}$ is a masa in $M_q$ as obtained in the previous section, to decide factoriality of $M_q$ in the case when $(U_t)$ has a non trivial fixed vector.

The arguments needed to prove factoriality of $M_q$ is divided into two cases, one dealing with the discrete part of the spectrum of $A$ corresponding to the eigen value 1 and the other dealing with the continuous part of the spectrum.

\begin{Theorem}\label{wmixingfactor}
Let $t\mapsto U_t$ be a strongly continuous orthogonal representation of $\R$ on a real Hilbert space $\CH_\R$. 
Suppose that the invariant subspace of weakly mixing vectors in $\CH_\R$ is non trivial. Then $M_q$ is a factor.
\end{Theorem}

\begin{proof}
Decompose $\CH_\R=\CH_c\oplus \CH_{wm}$ $($direct sum taken in $\langle\cdot,\cdot\rangle_{\CH_\C})$, where $\CH_c$ and $\CH_{wm}$ are closed invariant subspaces of the orthogonal representation consisting of compact and weakly mixing vectors respectively.  First of all note that $\CH_\R$ is infinite dimensional as $\CH_{wm}\neq 0$. If $\CH_c=0$, then by Eq. \eqref{modularaut} and Thm. \ref{CentraliserDescribe}, $(\sigma_t^{\varphi})$ acts ergodically on $M_q$. Consequently, $M_q$ is a $\rm{III}_1$ factor \cite{Ta2}. 

Let $\CH_c\neq 0$. Then $M_q^{\varphi}$ is non trivial from Thm. \ref{CentraliserDescribe}. Let $\xi\in \CH_{wm}$ be a unit analytic vector $($see Prop. \ref{analytic}$)$. Note that $\mathcal{Z}(M_q)\subseteq M_q^{\varphi}$. Borrowing notations from Thm. \ref{CentraliserDescribe} and the discussion in \S\ref{SectionCentralizer} preceding it, we have the following. For $\zeta_{i_j}\in \CS$, $1\leq i_j\leq N_1+2N_2$ $($or $1\leq i_j< N_1+2N_2$ as the case may be$)$ for $1\leq j\leq n$ and $\prod_{j=1}^n \beta_{i_j}=1$, note that $\zeta_{i_1}\otimes\cdots\otimes\zeta_{i_n}\in M_q^{\varphi}\Omega$. Note that the real and imaginary parts of $\zeta_{i_j}$ are analytic and individually orthogonal to $\xi$ with respect to $\langle\cdot,\cdot,\rangle_{U}$ and $\langle\cdot,\cdot\rangle_{\CH_\C}$ for all $1\leq j\leq n$. Then, decomposing vectors into real and imaginary parts and using  Eq. \eqref{Leftmult} and Lemma \ref{Wickformula}, it follows that $s_q(\xi)s_q(\zeta_{i_1}\otimes\cdots\otimes\zeta_{i_n})\Omega=\xi\otimes \zeta_{i_1}\otimes\cdots\otimes\zeta_{i_n}$,  while $s_q(\zeta_{i_1}\otimes\cdots\otimes\zeta_{i_n})s_q(\xi)\Omega=\zeta_{i_1}\otimes\cdots\otimes\zeta_{i_n}\otimes\xi$. This observation forces that if $a\in\mathcal{Z}(M_q)$, then $s_q(\xi)a\Omega =\xi\otimes a\Omega$, while $as_q(\xi)\Omega = a\Omega\otimes\xi$ using the  definition of right multiplication from \cite{Fal00} and the fact that $c_q(\xi),r_q(\xi)$ are  bounded. Thus, $s_q(\xi)$ cannot commute with $a$ unless $a$ is a scalar multiple of $1$, as $\Omega$ is a separating vector for $M_q$.  This completes the argument.
\end{proof}

\begin{Theorem}\label{Factor1eigenvalue}
Let $\CH_\R$ be a real Hilbert space with $dim(\CH_\R)\geq 2$. Let $t\mapsto U_t$ be a strongly continuous orthogonal representation of $\R$ on $\CH_\R$. Suppose there exists a unit vector $\xi_0\in\CH_\R$ such that  $U_t\xi_0=\xi_0$ for all $t\in \CH_\R$. Then $M_q$ is a factor.
\end{Theorem}

\begin{proof}
Let $x\in \mathcal{Z}(M_q)$. We will show that $x$ is a scalar multiple of $1$. By Thm. \ref{masa}, $M_{\xi_0}\subseteq M_q$ is a diffuse masa with a unique $\varphi$-preserving faithful normal conditional expectation. Thus, $\mathcal{Z}(M_q)\subseteq M_{\xi_0}$ and hence $x\in M_{\xi_0}$. As seen in the proof of Lemma \ref{Intersection}, $H_n^q(s_q(\xi_0))\Omega=\xi_0^{\otimes n}$ for all $n\geq 0$. Consequently, $x\Omega\in \overline{\text{span }\mathcal{E}_{\xi_0}}^{\norm{\cdot}_{q}}$ from Lemma \ref{Intersection} and hence 
\begin{align*}
x\Omega=\sum_{n=0}^{\infty} a_n\xi_{0}^{\otimes n}=\sum_{n=0}^{\infty} a_nH_n^q(s_q(\xi_0))\Omega, \text{  }a_n\in\C,
\end{align*}
where the series converges in $\norm{\cdot}_q$. 

Since $dim(\CH_\R)\geq 2$, so there exists an analytic vector $\xi_1\in \CH_\R$ $($see Prop. \ref{analytic}$)$ such that $\langle \xi_0,\xi_1\rangle_{\CH_\C}=0$. Hence, from Eq. \eqref{Leftmult} and Lemma \ref{Orthogonality} it follows that 
\begin{align*}
s_q(\xi_1)x\Omega &= \sum_{n=0}^{\infty} a_n s_q(\xi_1)H_n^q(s_q(\xi_0))\Omega\\
&= \sum_{n=0}^{\infty} a_n s_q(\xi_1)\xi_{0}^{\otimes n}=\sum_{n=0}^{\infty} a_n (\xi_1\otimes \xi_{0}^{\otimes n}).
\end{align*}
Again, from Eq. \eqref{Leftmult}, $H_n^q(s_q(\xi_0))s_q(\xi_1)\Omega=\xi_{0}^{\otimes n}\otimes \xi_1$ for all $n\geq 0$. To see this, we use induction. For $n=0$, the conclusion is obvious, and for $n=1$ the same follows from Lemma \ref{Orthogonality}. Assume that the result is true for $k=0,1,\cdots,n$. Note that the $q$-Hermite polynomials obey the following recurrence relations:
\begin{align*}
&H_0^q(x)=1, \text{ }H_1^q(x)=x \text{ and }\\
&xH_n^q(x)=H_{n+1}^q(x)+ [n]_q H_{n-1}^q(x), \text{ }n\geq 1, \text{ }x\in [-\frac{2}{\sqrt{1-q}},\frac{2}{\sqrt{1-q}}] \text{ \cite[Defn. 1.9]{BKS}}.
\end{align*}
Thus, by functional calculus one has 
\begin{align*}
H_{n+1}^q(s_q(\xi_0))\xi_1&=s_q(\xi_0)H_{n}^q(s_q(\xi_0))\xi_1 -[n]_q H_{n-1}^q(s_q(\xi_0))\xi_1\\
&=s_q(\xi_0)(\xi_0^{\otimes n}\otimes \xi_1) -[n]_q (\xi_0^{\otimes (n-1)}\otimes\xi_1)\\
&= \xi_0^{\otimes (n+1)}\otimes\xi_1, \text{ by Eq. \eqref{Leftmult} and Lemma \ref{Orthogonality}}.
\end{align*}
Thus, by induction the above conclusion follows. $($This can also be proved by Lemma \ref{Orthogonality} and Lemma \ref{Wickformula}$)$.  

Note that $x$ is a limit in $s.o.t$. of a sequence of operators from the linear span of $\{H_n^q(s_q(\xi_0)): n\geq 0\}$. Consequently, $xs_q(\xi_1)\Omega\in \overline{\text{span }\{\xi_{0}^{\otimes n}\otimes \xi_1:n\geq 0\}}^{\norm{\cdot}_{q}}$. Therefore, $xs_q(\xi_1)=s_q(\xi_1)x$ forces that $a_n=0$ for all $n\geq 1$. Thus, $x\Omega=a_0\Omega$ and hence $x=a_0 1$ as $\Omega$ is separating for $M_q$. So the proof is complete. 
\end{proof}

\begin{Remark}
Note that Hiai has proved that if the almost periodic part of $(U_t)$ is infinite dimensional, then the centralizer $M_q^{\varphi}$ of $M_q$ has trivial relative commutant and thus $M_q$ is a factor \cite{Hiai}. Thus, combined with our result the factoriality of $M_q$ remains open only when $\CH_\R$ is of even dimension and $(U_t)$ is ergodic $($and non trivial$)$. 
\end{Remark}

\section{Structure of the Centralizer}\label{StructureofCentralizer}
In this section, we discuss the factoriality of the centralizer $M_q^{\varphi}$ of the $q$-deformed Araki-Woods von Neumnn algebra $M_q$. By Rem. \ref{CentralizerRemark} and Thm. \ref{masa}, it follows that if the point spectrum of the analytic generator $A$ of $(U_t)$ is $\{1\}$ and is of simple multiplicity, then $M_q^{\varphi}$ is a masa in $M_q$. Thus, for the centralizer to be large, the almost periodic part of $(U_t)$ need to be reasonably large. 

For a short account on bicentralizers that follows, we refer the reader to \cite{Hag87}. Let $M$ be a separable type $\rm{III}_{1}$ factor and let $\psi$ be a faithful normal state on $M$. Denote $[x,y]=xy-yx$ and 
$[x,\psi]=x\psi-\psi x$ for $x,y\in M$. The asymptotic centralizer of $\psi$ is defined to be
\begin{align*}
\text{AC}_{\psi}=\{(x_n)\in \ell^{\infty}(\N,M): \norm{[x_n,\psi]}\rightarrow 0 \text{ as }n\rightarrow \infty\}.
\end{align*}
Observe that $\text{AC}_{\psi}$ is a unital $C^{\ast}$-subalgebra of $\ell^{\infty}(\N,M)$. The bicentralizer of $\psi$ is defined by 
\begin{align*}
B_{\psi}=\{y\in M: [y,x_n]\rightarrow 0 \text{ ultrastrongly as }n\rightarrow\infty \text{ for all }(x_n)\in \text{AC}_{\psi}\}.
\end{align*}
Note that $B_{\psi}$ is a von Neumann subalgebra of $M$ which is globally invariant with respect to the modular automorphism group $(\sigma_t^{\psi})$. Further, $B_{\psi}\subseteq (M^{\psi})^{\prime}\cap M$. The type $\rm{III}_{1}$ factor $M$ is said to have trivial bicentralizer if $B_{\psi}=\C 1$ for any faithful normal state $\psi$ of $M$. The bicentralizer problem of Connes is open and asks if every separable type $\rm{III}_{1}$ factor has trivial bicentralizer. 

\begin{Theorem}\label{TrivialRelativeCommutant}
Let $\CH_\R$  be a real Hilbert space such that $dim(\CH_\R)\geq 2$. Let $(U_t)$ be a strongly continuous real orthogonal representation of $\R$ on $\CH_\R$ such that -\\
\noindent $(i)$ there exists a unit vector $\xi_0 \in \CH_\R$  satisfying $U_t\xi_0= \xi_0$ for all $t \in \R$, \\
\noindent$(ii)$ the almost periodic part of $(U_t)$ is at least two dimensional.\\
Then, 
\begin{align*}
(M_q^\varphi)^{\prime} \cap M_q = \C 1.
\end{align*}
In particular, the centralizer $M_q^\varphi$ of $M_q$ is a factor. Moreover, if $M_q$ is a $\rm{III}_{1}$ factor then it has trivial bicentralizer.
\end{Theorem}

\begin{proof}
The first statement was settled in the case when the almost periodic part of $(U_t)$ is two dimensional $($see Cor. \ref{Trivialrelcommutant}$)$. First of all, note that from Cor. \ref{Corollary}, the von Neumann algebra $M_{\xi_0} =vN(s_q(\xi_0)) \subseteq M_q^\varphi$ is a masa in $M_q$ with a unique $\varphi$-preserving faithful normal conditional expectation $\E_{\xi_0}: M_q \rightarrow M_{\xi_0}$. Therefore, $(M_q^\varphi)^{\prime}\cap M_q \subseteq M_{\xi_0}$. Let $x\in (M_q^{\varphi})^{\prime}\cap M_q$. 

Since the dimension of the almost periodic part of $(U_t)$ is at least two, so from Thm. \ref{CentraliserDescribe}, it follows that there exist vectors $\zeta_i\in \CH_\C$ $($with real and imaginary parts individually analytic$)$, $1\leq i\leq k$, such that $\zeta_1\otimes\cdots\otimes\zeta_{k}\in M_q^{\varphi}\Omega$ and $\zeta_{i}$ and as well as its real and imaginary parts are orthogonal to $\xi_0$ for all $1\leq i\leq k$, with respect to $\langle\cdot,\cdot\rangle_{\CH_\C}$ $($as well as orthogonal in $\langle\cdot,\cdot\rangle_{U}$, as $dim(\CH_\R)\geq 2)$. Let $y=s_q(\zeta_1\otimes\cdots\otimes\zeta_{k})\in M_q^{\varphi}$. As seen in the proof of Lemma \ref{Intersection}, $H_n^q(s_q(\xi_0))\Omega=\xi_0^{\otimes n}$ for all $n\geq 0$. Consequently, $x\Omega\in \overline{span\text{ }\mathcal{E}_{\xi_0}}^{\norm{\cdot}_{q}}$ from Lemma \ref{Intersection} and hence, 
\begin{align*}
x\Omega=\sum_{n=0}^{\infty} a_n\xi_{0}^{\otimes n}=\sum_{n=0}^{\infty} a_nH_n^q(s_q(\xi_0))\Omega, \text{  }a_n\in\C,
\end{align*}
where the series converges in $\norm{\cdot}_q$. Moreover, decomposing vectors into real and imaginary parts and using Lemma \ref{Wickformula}, it follows that
\begin{align*}
&yx\Omega \in \overline{span\text{ }\{\zeta_1\otimes\cdots\otimes\zeta_{k}\otimes \xi_0^{\otimes n}: \text{ } n\geq 0 \}}^{ \norm{\cdot}_q }.
\end{align*}
Further, decomposing vectors into real and imaginary parts and using Eq. \eqref{Leftmult} and Lemma \ref{Orthogonality} it follows that $s_q(\xi_0)(\zeta_1\otimes\cdots\otimes\zeta_{k})=\xi_0\otimes \zeta_1\otimes\cdots\otimes\zeta_{k}$. Assume that $s_q(H_m^q(\xi_0))(\zeta_1\otimes\cdots\otimes\zeta_{k})=\xi_0^{\otimes m}\otimes \zeta_1\otimes\cdots\otimes\zeta_{k}$,  for $m=0,1,\cdots,n$. Using the recurrence relations of $q$-Hermite polynomials $($as in the proof of Thm. \ref{Factor1eigenvalue}$)$,
Eq. \eqref{Leftmult},  Lemma \ref{Orthogonality} and the induction hypothesis, it follows that $H_n^q(s_q(\xi_0))(\zeta_1\otimes\cdots\otimes\zeta_{k})=\xi_0^{\otimes n}\otimes \zeta_1\otimes\cdots\otimes\zeta_{k}$, for all $n\geq 0$. Now note that 
\begin{align*}
xy\Omega &= Jy^{*}J x\Omega=\sum_{n=0}^{\infty} a_n Jy^*J H_n^q(s_q(\xi_0))\Omega=\sum_{n=0}^{\infty} a_n H_n^q(s_q(\xi_0))y\Omega\\
&=\sum_{n=0}^{\infty} a_n \big(\xi_0^{\otimes n}\otimes (\zeta_1\otimes\cdots\otimes\zeta_{k})\big).
\end{align*}
Since, $xy = yx$, so $a_n=0$ for all $n\neq 0$. Thus, the first statement follows. 

The final statement is a direct consequence of Connes-St{\o}rmer transitivity theorem \cite{CS78}.
\end{proof}

\section{Type Classification}

In this section, we describe the type of $M_q$ under the same constraints as in \S\ref{Factoriality} by showing that the type depends on the spectral information of $A$ as expected. To begin with, we recall some well known facts about the $S$ invariant of Connes.

The $S$ invariant of a factor $M$ was defined in \cite{Co73} to be the intersection
over all faithful normal semifinite $($f.n.s.$)$ weights $\phi$ of the spectra of the associated modular operators
$\Delta_\phi$. Further, $M$ is a type $\rm{III}$ factor if and only if $0 \in S(M)$ and in this case Connes classified type $\rm{III}$ factors using their $S$ invariant as follows:
\begin{align*}
S(M) = 
\begin{cases} 
[0, \infty),  &\text{ if } M  \text{ is  type } \rm{III}_1,\\
\{ 0, 1\},   &\text{ if } M  \text{ is  type } \rm{III}_0,\\
\{ \lambda^n : \text{ } n \in \Z \}\cup \{ 0\},  &\text{ if } M  \text{ is  type } \rm{III}_\lambda, \text{ }0<\lambda<1. 
\end{cases}
\end{align*}

Also, recall from \cite{Co73} that for a fixed faithful normal state $($resp. f.n.s. weight$)$ $\phi$ on $M$, the $S$ invariant can be written as
\begin{align*} 
S(M) = \cap \{ \text{Sp}(\Delta_{\phi_p}): \text{ } 0 \neq p \in \CP(\CZ(M^\phi))\},
\end{align*}
$\CP(\CZ(M^\phi))$ denoting the lattice of projections in the center of the centralizer $M^{\phi}$ and $\phi_p =\phi_{|pMp}$. So, let $\phi$ be a faithful normal state on $M$ and let $0\neq p \in M^{\phi}$ be a projection. Let $\Delta_{\phi_p}$ and $(\sigma_t^{\phi_p})$ respectively denote the modular operator and the modular automorphism group of the corner $pMp$ associated to the positive functional $\phi_p$. When $p=1$, write $\Delta_{\phi_1}$ and $(\sigma_t^{\varphi_1})$ respectively as $\Delta_\phi$ and $(\sigma_t^{\phi})$. It is clear that $\sigma_t^{\phi_p}(pxp)=p\sigma_t^{\phi}(x)p$ for all $x\in M$ and $t\in \R$. It is also easy to check that $\sigma_t^{\phi_p}$ is implemented by $\Delta_{\phi_p}^{it}=p\Delta_{\phi}^{it} p$ for all $t\in \R$.

\begin{Theorem}\label{WeakmixingtoIII_1}
Let $(U_t)$ be a strongly continuous real orthogonal representation of $\R$ on a real Hilbert space $\CH_\R$ such that the weakly mixing component of $(U_t)$  is non trivial. Then $M_q$ is a type $\rm{III}_1$ factor. 
\end{Theorem}

\begin{proof}
Note that $\CH_\R$ is infinite dimensional. We need to show that $S(M_q)=[0,\infty)$. So, let $0 \neq p \in \CP(\CZ(M_q^\varphi))$. By the hypothesis and Prop. \ref{Eigenvector}, there exists 
$0\neq \xi \in \CH_\R\subseteq \CH_\C\subseteq \mathcal{F}_q(\CH)$ such that
\begin{align*}
\frac{1}{2T}\int_{-T}^{T}\abs{\langle U_t\xi,\xi\rangle_U}^{2}dt\rightarrow 0, \text{ as }T\rightarrow \infty \text{ (see Eq. \eqref{Liftisunitary})}.
\end{align*}
Thus, by Eq. \eqref{qFock}, Eq. \eqref{modulartheory}, Eq. \eqref{modularaut} and the discussion following it, one has  
\begin{align*}
\frac{1}{2T}\int_{-T}^{T}\abs{\langle \mathcal{F}(U_t)\xi,\xi\rangle_q}^{2}dt\rightarrow 0, \text{ as }T\rightarrow \infty.
\end{align*}
Consequently, if $\mu_{\xi}$ denotes the elementary spectral measure $($on $\R)$ associated to $\xi$ of the representation $\{t\mapsto\mathcal{F}(U_t):t\in \R\}$, then $\mu_{\xi}$ is non atomic $($from Eq. \eqref{modularaut}$)$. 

If $p\neq 1$, note that $p\xi, (1-p)\xi$ are non zero vectors. Indeed, if $\zeta\in M_q^{\varphi}\Omega$ is such that $s_q(\zeta)=p$ $($see Eq. \eqref{sq}$)$, then by Thm. \ref{Wickformula} and Thm. \ref{CentraliserDescribe} $($as in the proof of Thm. \ref{wmixingfactor}$)$, it follows that $p\xi =\zeta\otimes \xi\neq 0$. Similar is the argument for $(1-p)\xi$. Let $\mu_{p\xi},\mu_{(1-p)\xi}$ respectively denote the elementary spectral measures of $\{t\mapsto\mathcal{F}(U_t):t\in \R\}$ associated to the vectors $p\xi$ and $(1-p)\xi$. Note that $\mu_{p\xi}$ is the elementary spectral measure of $t\mapsto p\Delta^{it} p$ $(=\Delta_{\varphi_p}^{it})$, $t\in\R$, corresponding to the vector $p\xi$, and the former implements  $(\sigma_t^{\varphi_p})$. Also, as $p\in M_q^{\varphi}$, so the range of $p$ is an invariant subspace of $\{\mathcal{F}(U_t):t\in \R\}$. Hence,
\begin{align*}
\langle \mathcal{F}(U_t)p\xi,(1-p)\xi\rangle_q =0, \text{ for all }t\in \R.
\end{align*}
Consequently, $\mu_{\xi}=\mu_{p\xi}+\mu_{(1-p)\xi}$, thus $\mu_{p\xi}$ and $\mu_{(1-p)\xi}$ are both non atomic. 

Note that the weakly mixing component of $\{t\mapsto \mathcal{F}(U_t):t\in\R\}$ is invariant under the anti-unitary $J$. This follows by using the fact that $J\Delta^{it} J=\Delta^{it}$ for all $t\in \R$ and by the definition of weak mixing. Thus, $\mu_{p J\xi}$ is non zero and non atomic. Note that both $\xi$ and $J\xi$ are vectors in the $1$-particle space $\CH$ of $\mathcal{F}_q(\CH)$. This forces that the spectral measure of the action $\{t\mapsto \mathcal{F}(U_t):t\in \R\}$ when restricted to the $1$-particle space $\CH$ contains a non trivial non atomic component on both sides of $0$ by an application of Stone-Weierstrass theorem. Since, $\mathcal{F}(U_t)=id\oplus \oplus_{n\geq 1}U_t^{\otimes_q n}$, $t\in\R$, it follows that $\text{Sp}(\Delta_{\varphi_p})$ contains a closed multiplicative group inside $[0,\infty)$ generated by the support of a non atomic measure such that the support intersects both $(0,1)$ and $(1,\infty)$ non trivially $($in measure theoretic sense$)$. So, $\text{Sp}(\Delta_{\varphi_p})=[0,\infty)$. Thus, the result follows.
\end{proof}

Now we turn to the case when the orthogonal representation is almost periodic.

\begin{Theorem}\label{AlmostPeriodicCase}
Let $(U_t)$ be a strongly continuous almost periodic orthogonal representation of $\R$ on a real Hilbert space $\CH_\R$ such that $dim(\CH_\R)  \geq 2$ and such that there exists a unit vector $\xi_0\in\CH_\R$ with $U_t\xi_0=\xi_0$ for all $t\in\CH_\R$. Let $G$ be the closed subgroup of $\R^{\times}_+ $
generated by the spectrum of $A$. Then,
\begin{align*}
M_q \text{ is }  
\begin{cases} 
\text{ type } \rm{III}_1 &\text{ if } G = \R_+^{\times},\\
\text{ type } \rm{III}_\lambda &\text{ if } G = \lambda^{\Z}, \text{ } 0 < \lambda < 1,\\
\text{ type } \rm{II}_1 &\text{ if } G = \{1\}.
\end{cases}
\end{align*}  
The type $\rm{II}_1$ case corresponds to $(U_t)=(id)$ and thus $M_q$ is the Bo$\overset{.}{\text{z}}$ejko-Speicher's $\rm{II}_1$ factor.
\end{Theorem}

\begin{proof} 
The hypothesis forces that if $dim(\CH_\R)=2$, then $M_q$ is a $\rm{II}_1$ factor from Cor. \ref{Trivialrelcommutant} and there is nothing to prove. If $dim(\CH_\R)\geq 3$, then by Thm. \ref{TrivialRelativeCommutant} it follows that $(M_q^{\varphi})^{\prime}\cap M_q=\C 1$. Thus, $M_q^{\varphi}$ is a factor, and hence $S(M_q)$ is completely determined by $\text{Sp}(\Delta)$. Now use the fact that $\mathcal{F}(U_t)=id\oplus \oplus_{n\geq 1}U_t^{\otimes_q n}$, $t\in\R$, and Prop. \ref{Eigenvalueinamplification} to complete the proof. We omit the details.
\end{proof}


\end{document}